\documentclass[10pt]{amsart}
\usepackage{amsthm, amsfonts, amssymb, amsmath, bbold, url, stmaryrd, graphicx, xcolor, enumitem}
\usepackage{avant}

\newcommand{\C}{\mathbb{C}}
\newcommand{\R}{\mathbb{R}}

\newcommand{\Z}{\mathbb{Z}}

\newcommand{\F}{\mathbb{F}}
\newcommand{\M}{\mathcal{M}}

\newcommand{\e}{\epsilon}

\newcommand{\Fp}{\F_{\!p}}

\usepackage[skip=4pt]{caption}

\newtheorem{thm}{Theorem}[section]
\newtheorem{lem}[thm]{Lemma}
\newtheorem{countlem}[thm]{Counting Lemma}
\newtheorem{disclem}[thm]{Discriminant Lemma}
\newtheorem{cor}[thm]{Corollary}

\theoremstyle{definition}

\newtheorem*{defn}{Definition}
\newtheorem*{rem}{Remark}

\newtheorem*{ack}{Acknowledgments}
\newtheorem*{AI}{Declaration of AI use}

\numberwithin{equation}{section}

\begin{document}
\title{Counting in Vieta graphs over $\Fp$}
\keywords{Generalized Markoff graph, Vieta graph, Jacobsthal sum.}
\subjclass{Primary: 11D25. Secondary: 05C25, 11L10.}

\date{\today}
\author{Bogdan Nica}

\begin{abstract} 
We introduce and study a finite simple graph of algebraic origin: the Vieta graph on the solution set over $\Fp$ to a symmetric, multivariate equation which is quadratic in each variable. This construction is a broad generalization of the Markoff graph over $\Fp$, extensively studied in the recent literature. We give a systematic approach, partly based on quadratic character sums, to the following basic counting questions: how many vertices does a Vieta graph have, and what is the degree distribution? We focus on explicit counts, addressing the low-dimensional cases in three and four variables. 
\end{abstract}

\address{\newline Department of Mathematical Sciences \newline Indiana University Indianapolis}
\email{bnica@iu.edu}

\maketitle

\setcounter{tocdepth}{1}
\tableofcontents

\part{Introduction}
Let $p\geq 5$ be a prime, and let $\Fp$ be the finite field with $p$ elements.

\section{The Markoff graph over $\Fp$} Consider the equation 
\begin{align}\label{eq: M}
x^2+y^2+z^2=xyz.
\end{align}
If $(x,y,z)$ is a solution to \eqref{eq: M} then so are its \emph{Vieta flips} $(yz-x,y,z)$, $(x,xz-y,z)$, $(x,y,xy-z)$. The \emph{Markoff graph over $\Fp$} is constructed is follows: the vertices are the solutions $(x,y,z)\in \Fp^3$ to the equation \eqref{eq: M}; edges connect two distinct vertices if one is a Vieta flip of the other. We note that 
edges are undirected, since the Vieta flips define involutory maps. There are no multiple edges, and no loops. In short the Markoff graph over $\Fp$ is a finite simple graph. 

The name of the graph owes to the fact that \eqref{eq: M} is equivalent, via the uniform rescalings $x:=3x$, $y:=3y$, $z:=3z$, to the equation $x^2+y^2+z^2=3xyz$. As a diophantine equation, this is the well-known Markov equation. Both transliterations--Markov and Markoff--are in use, though the latter one has stuck for the graph we are interested in. In what follows, we refer to \eqref{eq: M} as the \emph{Markoff cubic}. 

The Markoff graph over $\Fp$ has the trivial solution $(0,0,0)$ as an isolated vertex. The graph obtained by discarding the vertex $(0,0,0)$ is usually called the \emph{Markoff graph mod $p$} in the literature. A conjecture of Baragar \cite{Bar} from 1991 amounts to the statement that the Markoff graph mod $p$ is connected. A decade ago, Bourgain, Gamburd, and Sarnak \cite{BGS} have reignited interest in the connectivity conjecture, by recasting it as a `strong approximation' property. The connectivity conjecture is now known to be true for all but finitely many $p$, thanks to the remarkable work of Bourgain--Gamburd--Sarnak \cite{BGS2} and Chen \cite{Chen}; see also Martin \cite{Mar}. Further recent literature on the Markoff graph mod $p$ includes de Courcy-Ireland \cite{Cou} and Eddy--Fuchs--Litman--Martin--Tripeny \cite{EFL}. We especially refer to Silverman \cite{Sil} and Gamburd \cite{Gam} for a short, respectively a long, overview of the Markov equation and various structures related to it--including the Markoff graph mod $p$.

\begin{figure}[ht]
    \hspace*{-0.6cm}
    \includegraphics[width=1.1\textwidth]{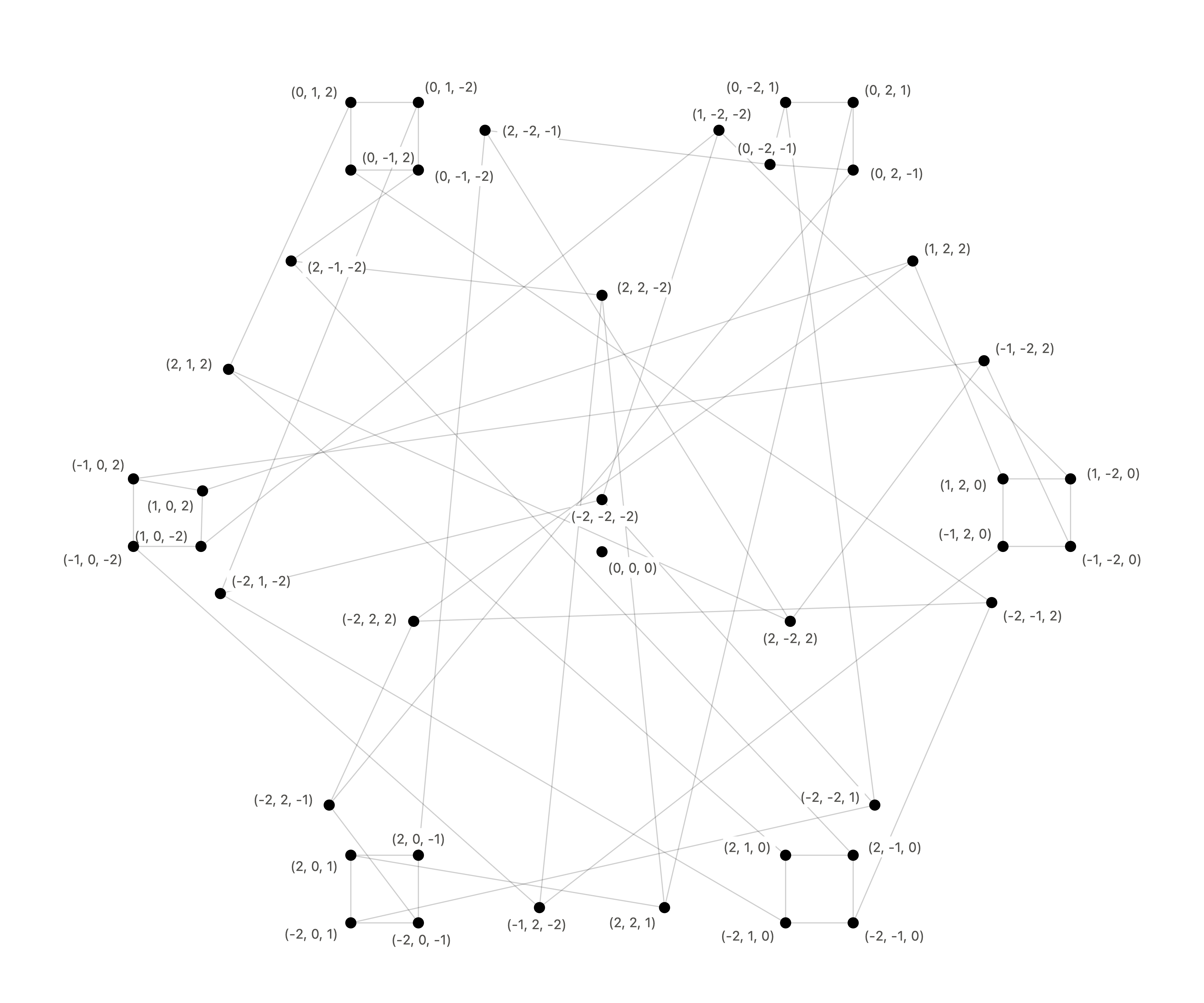}
    \caption{The Markoff graph over $\F_{\!5}$. There are $41$ vertices, all of degree $3$ except for the isolated vertex $(0,0,0)$.}
    \label{fig: C0_D0}
\end{figure}

Some very recent works have studied Markoff-type graphs over $\Fp$, defined by equations which generalize \eqref{eq: M}. The most immediate generalization is the equation
\begin{align}\label{eq: G0}
x^2+y^2+z^2=xyz+D.
\end{align}
A very recent work of Satake and Yamasaki \cite{SY} investigates the Markoff-type graph over $\Fp$ defined by \eqref{eq: G0}. Previously, de Courcy-Ireland and Lee \cite{CIL} have studied a related graph arising from \eqref{eq: G0}, in which the edges are defined by certain `twisted' Vieta flips.

There is a rich mathematical landscape around \eqref{eq: G0}, when viewed as an equation over $\R$ or $\C$. Indeed, equation \eqref{eq: G0} is related to the Fricke identity for traces of $\mathrm{SL}_2$ matrices, and to the character variety of a torus with one puncture; see, for instance, Goldman's survey \cite{Gold}. Diophantine aspects of \eqref{eq: G0} have been studied by Mordell \cite{Mor} and, more recently, by Ghosh--Sarnak \cite{GS}. Besides $D=0$, which gives the Markoff cubic, another distinguished value of the parameter $D$ in \eqref{eq: G0} is $D=4$; this defines the \emph{Cayley cubic}
\begin{align}\label{eq: G4}
x^2+y^2+z^2=xyz+4.
\end{align}

Another Markoff-type graph over $\Fp$ was considered by de Courcy-Ireland, Litman, and Mizuno \cite{CLM}. It is defined by the equation
 \begin{align}\label{eq: GCLM}
x^2+y^2+z^2+K_1yz+K_2xz+K_3xy=(3+K_1+K_2+K_3)xyz
\end{align}
where $K_1+K_2+K_3\neq -3$. For $K_1=K_2=K_3=0$, we recover the Markov equation. As a diophantine equation, \eqref{eq: GCLM} has been studied by Gyoda and Matsushita \cite{GM}. More recently, the symmetric case of \eqref{eq: GCLM}, namely
 \begin{align}\label{eq: GSLM}
x^2+y^2+z^2+K(xy+yz+zx)=(3+3K)xyz
\end{align}
where $K\neq -1$, has received attention--again, as a diophantine equation--in work of Gyoda--Maruyama \cite{GM2}, Gyoda--Maruyama--Sato \cite{GM3}, and Banaian \cite{Ban}. It appears that \eqref{eq: GSLM} is a structurally sound generalization of the Markov equation $x^2+y^2+z^2=3xyz$.

\section{Vieta graphs over $\Fp$}
 The algebraic mechanism underlying the Markoff and the Markoff-type graphs can be broadly formalized as follows: the solution set of a symmetric equation over $\Fp$, which is quadratic in each variable, can be given the structure of a finite simple graph by introducing edges which reflect the Vieta relation between the roots of a quadratic equation. We now discuss this in detail.
 
Let $f(x_1,\dots,x_n)$ be a symmetric polynomial, with coefficients in $\Fp$, which has degree $2$ in each variable. Let $V(f)\subseteq \Fp^n$ denote the algebraic variety defined by $f$, that is to say, the solution set to the equation $f(x_1,\dots,x_n)=0$. Besides the obvious permutational symmetries, the variety $V(f)$ admits $n$ involutions, which arise as follows. Fix $i\in \{1,\dots, n\}$. We write the definition equation $f(x_1,\dots,x_n)=0$ as a quadratic equation in $x_i$:
\begin{align}\label{eq: mu2}
\alpha_i x_i^2+\beta_i x_i+\gamma_i=0,
 \end{align}
where $\alpha_i$, $\beta_i$, $\gamma_i$ are symmetric polynomials in $x_1,\dots, x_{i-1},x_{i+1}, \dots, x_n$. In fact these can be expressed in terms of $f$ and its partial derivatives as 
\begin{align}
\alpha_i=\tfrac{1}{2}\partial^2_i f, \qquad \beta_i=\partial_i f-x_i\partial^2_i f, \qquad \gamma_i=f|_{x_i=0}.
\end{align}
Now if $x=(x_1,\dots, x_i, \dots ,x_n)$ is a solution to \eqref{eq: mu2} then so is $(x_1,\dots, x'_i,\dots ,x_n)$, where $x_i'$ is the other root of \eqref{eq: mu2}. Explicitly, the (additive) Vieta relation gives
\begin{align}\label{eq: mu3}
x_i'=x_i-\frac{(\partial_i f)(x)}{\tfrac{1}{2}(\partial^2_i f)(x)}.
 \end{align}
Formula \eqref{eq: mu3} is only partially defined. The issue comes from `irregular' solutions $x$ for which $(\partial^2_i f)(x)=0$. Over regular solutions, the root-flipping $x_i\mapsto x_i'$ is an involution. We may thus define an involution, the \emph{$i$-th Vieta flip}, on $V(f)$ by
\begin{align}
x=(x_1,\dots,x_i,\dots,x_n) \mapsto 
\begin{cases} (x_1,\dots,x'_i,\dots,x_n) & \textrm{if } (\partial^2_i f)(x)\neq 0,\\
x & \textrm{if } (\partial^2_i f)(x)=0.
\end{cases}
\end{align}

By using the $n$ Vieta flips obtained in this way, we turn the algebraic variety $V(f)$ into an unoriented graph. More precisely, it will be a finite simple graph--without loops or multiple edges.
 
 \begin{defn}
The \emph{Vieta graph} defined by the equation $f(x_1,\dots,x_n)=0$ is the graph with vertex set $V(f)$, in which two distinct vertices are joined by an edge whenever one is a Vieta flip of the other.
\end{defn}

We often use an alternate, more concise phrasing--the Vieta graph of $f$. We refer to the number of variables, $n$, as the dimension of the Vieta graph of $f$.
 
In this paper, we address the general problem of counting the number of vertices, and the degree distribution in Vieta graphs. Counting the number of \emph{vertices} in the Vieta graph of $f$, in other words the number of solutions in $\Fp^n$ to the polynomial equation $f(x_1,\dots,x_n)=0$, is a classical problem. But the Vieta structure on the solution set naturally brings in a plethora of new aspects. Counting the number of \emph{edges} in the Vieta graph of $f$ is a seed question which immediately calls for the degree distribution, that is to say, the number of vertices of a given degree in the Vieta graph of $f$. Our specific goal in these counting questions is to provide \emph{explicit} computations. But where to look? Beyond the Markoff equation, which other equations are interesting? This was not obvious to us from the outset. One upshot of the quest for explicit computations is that it uncovers and certifies equations that are, we believe, worthy of further study. 

Explicit counts for the number of solutions to a polynomial equation of the form $f(x_1,\dots,x_n)=0$, having degree $2$ in each variable, has been an on-going theme in Carlitz's extensive work. Especially relevant for us are his counting results from \cite{Car1, Car2, Car3}. Herein, these results are refined and interwoven into the intricate tapestry of counting in Vieta graphs. We will see, in fact, that explicit edge-counting is usually much more intricate than explicit vertex-counting--which is what Carlitz's results mean in the Vieta context. Explicit counting in Vieta graphs often draws from explicit computations of character sums; this is the particular angle that is appealing to us, and we will use it intensively in this paper. In particular, we will eventually rely on Jacobsthal sums as an important technical ingredient.

In our search for explicit computations, we narrow our focus to \emph{regular} equations. By this we mean that every solution in $V(f)$ is regular in the sense discussed above, that is to say, the second derivatives $\partial_1^2f, \dots, \partial_n^2f$ never vanish on the variety $V(f)$. This streamlines the definition of Vieta flips, for \eqref{eq: mu3} is now fully defined over $V(f)$. The simplest way to ensure regularity is to take 
\begin{align}\label{eq: mu1}
f(x_1,\dots,x_n)=x_1^2+\dots+x_n^2-\ell(x_1,\dots,x_n)
 \end{align}
where $\ell(x_1,\dots,x_n)$ is a symmetric polynomial which has degree $1$ in each variable. Indeed, in this case we have $\partial_i^2f=2$ for each $i$. Formula \eqref{eq: mu3} becomes
\begin{align*}
x_i'=(\partial_i\ell)(x_1,\dots,x_{i-1}, x_{i+1}, \dots,x_n) -x_i.
\end{align*}

In this form we are drawing closer to the Markoff graph; yet, as we will see, there is still enough generality left. All equations considered herein are in fact of the form \eqref{eq: mu1}. Let us mention, at this point, an interesting irregular equation, taken from \cite{FLST}. The equation--in fact, a parametric family--is given by 
\begin{align}\label{eq: A}
x^2+y^2+z^2+x^2y^2z^2=Axyz,
\end{align}
where $A\neq 0$. Irregularity occurs when $p\equiv 1$ mod $4$. We propose the problem of counting the number of vertices and the degree distribution in the Vieta graph associated to the equation \eqref{eq: A}.  

Obviously, Vieta graphs can be defined over any finite field whose characteristic is odd, not just over a prime field $\Fp$. We keep to $\Fp$ for the sake of connecting with the diophantine origin of some of the underlying equations, but also for the sake of simplicity--a handful of the explicit computations obtained herein take a more concise form over $\Fp$. 

The landscape of Vieta graphs invites many questions. The following stand out to us at this point:
\begin{itemize}[leftmargin=2.5em]
\item[$\bullet$] find (parametric families of) equations for which the number of vertices and the degree distribution in the corresponding Vieta graphs can be computed explicitly;
\item[$\bullet$] give asymptotic estimates for the number of vertices and the degree distribution in Vieta graphs;
\item[$\bullet$] study the connectivity of Vieta graphs--more precisely, count and classify the connected components;
\item[$\bullet$] investigate the existence and the count of small substructures in Vieta graphs: the complete graphs $K_3$, $K_4$, or $K_5$; the complete bipartite graphs $K_{2,2}$, $K_{2,3}$, or $K_{3,3}$; the short cycles $C_4$, $C_5$, or $C_6$;
\item[$\bullet$] study spectral aspects of Vieta graphs--chiefly eigenvalues, but eigenvectors as well--from the Laplacian viewpoint, which we believe to be more manageable than the adjacency viewpoint.
\end{itemize}

The connectivity question has been the most studied aspect--very intensely for the Markoff graph, and more recently in \cite{CLM} for the Markoff-type graph defined by \eqref{eq: GCLM}. Non-planarity of the Markoff graph is settled in \cite{Cou}. Non-planarity and instances of the small-structures question for the graph defined by equation \eqref{eq: G0} are studied in \cite{SY}. All these results are confined to dimension $3$, for certain Markoff-type cubics.

This paper takes a different direction from has been done in the recent literature, by  focusing on first question. Our study is fairly thorough in dimension $3$, but only fragmentary in dimension $4$. At any rate, our explicit counting results--in two different dimensions, for both cubics and quartics--suggest that the framework of Vieta graphs over $\Fp$ is viable and intriguing.

Here is the plan of the paper. Part 2 is devoted to some counting generalities that are applicable in any dimension. The bulk of the paper is in Part 3. There we discuss in detail Vieta graphs in dimension $3$ that are defined by the most general Markoff-type cubic. We identify several parametric families for which we can achieve our goal of explicit computations. In Part 4, we obtain explicit computations for two parametric families of Vieta graphs in dimension $4$ that are defined by quartics. 

\begin{ack}
I thank Roland Roeder, for informative discussions on equation \eqref{eq: GNS} and general feedback, and Ryan Eades, for lively conversations during the initial stage of this work.
\end{ack}

\begin{AI}
This work was conceived and written without use of artificial intelligence, except for the Figures which were rendered using Claude. Claude was then used again for a review of the completed draft. The final version incorporates several minor edits arising from Claude's feedback. 
\end{AI}

\bigskip
\part{Generalities}

\section{The $N_k$ counts}
Consider a Vieta graph in dimension $n$. Since there are $n$ Vieta flips, each vertex has degree at most $n$; in fact, one expects generic vertices to have degree equal to $n$. But one also expects `deficient' vertices, whose degree is less than $n$. Our goal is to determine the degree distribution in a Vieta graph: to count, for any given $d\in \{0,1,\dots, n\}$, the number of vertices of degree $d$.

Degree deficiencies occur when vertices are fixed by Vieta flips. If $f(x_1,\dots,x_n)$ is the underlying polynomial and $k\in \{0,\dots,n\}$, we let 
\begin{align}
N_k(f)=\# \Big\{x\in V(f): x \textrm{ is fixed by at least } k \textrm{ Vieta flips}\Big\}. 
\end{align}
This is well-defined, in view of the symmetry. It is obvious from the definition that the sequence of $N_k$ counts is non-increasing: $N_0(f)\geq N_1(f)\geq \dots \geq N_n(f)$. 

The $N_k$ counts underlie the degree distribution in the Vieta graph of $f$. Clearly, $N_0(f)$ counts all the vertices, while $N_n(f)$ counts the isolated vertices. On the other hand, $nN_1(f)$ overcounts the deficient vertices. In fact, by the inclusion-exclusion principle, we get the following.

\begin{lem}\label{lem: Nk}
Let $d\in \{0,1,\dots, n\}$. Then the number of vertices of degree $d$ in the Vieta graph of $f$ is given by
\begin{align}
\binom{n}{d}\sum_{k=0}^d (-1)^k \binom{d}{k} N_{n-d+k}(f).
\end{align}
In particular, there are 
\begin{align}
\sum_{k=1}^n (-1)^{k-1} \binom{n}{k} N_{k}(f)
\end{align}
deficient vertices among the $N_0(f)$ vertices of the Vieta graph of $f$.
\end{lem}

From Lemma~\ref{lem: Nk} we quickly infer the following.

\begin{cor}\label{cor: reg}
No vertex in the Vieta graph of $f$ has degree less than $d$ if and only if $N_{n-d+1}(f)=0$. In particular, the Vieta graph of $f$ is regular if and only if $N_1(f)=0$.
\end{cor} 

A more involved by-product of the degree distribution is the edge count in the Vieta graph. Interestingly, it turns out that only $N_0(f)$ and $N_1(f)$ are involved in this count.

\begin{cor}
The number of edges in the Vieta graph of $f$ is given by
\begin{align*}
\frac{n}{2}\big(N_0(f)-N_1(f)\big).
\end{align*}
\end{cor}

\begin{proof}
Let $\# E$ denote the number of edges. The handshake lemma gives
\begin{align*}
2\cdot \# E&=\sum_{d=0}^n d\cdot (\textrm{number of vertices of degree }d)\\
&=\sum_{d=0}^n d \binom{n}{d}\sum_{k=0}^d (-1)^k \binom{d}{k} N_{n-d+k}\\
&=\sum_{i=0}^n\bigg(\sum_{k=0}^{n-i}  (k+i) (-1)^k\binom{n}{k+i} \binom{k+i}{k} \bigg)N_{n-i}
\end{align*}
by setting $d:=k+i$ in the last step. Next, let us compute the parenthesized coefficient of $N_{n-i}$, which we temporarily denote by $c_{n-i}$. As 
\[\binom{n}{k+i} \binom{k+i}{k}=\binom{n}{i} \binom{n-i}{k}\]
we have
\[c_{n-i}=\binom{n}{i}\sum_{k=0}^{n-i}  (k+i) (-1)^k \binom{n-i}{k}.\]
Now
\[\sum_{k=0}^{n-i}  i (-1)^k \binom{n-i}{k}=i (1-1)^{n-i}=i\cdot \llbracket i=n\rrbracket,\]
and likewise
\[\sum_{k=0}^{n-i}  k (-1)^k \binom{n-i}{k}=\sum_{k=1}^{n-i}  (n-i)(-1)^k \binom{n-i-1}{k-1}=-(n-i)\cdot \llbracket i=n-1\rrbracket.\]
Therefore
\[c_{n-i}=\binom{n}{i}\Big(i\cdot \llbracket i=n\rrbracket-(n-i)\cdot \llbracket i=n-1\rrbracket\Big)
=\begin{cases}n & \textrm{ if } i=n,\\ -n & \textrm{ if } i=n-1,\\0 & \textrm{ if } i\leq n-2. \end{cases}\]
To conclude, 
\[2\cdot \# E=\sum_{i=0}^n c_{n-i}N_{n-i}=n N_0-nN_1.\]
The claimed formula for $\#E$ follows by rearranging.
\end{proof} 

The above counts are quite general--they hold whenever we construct a simple graph by using involutions. The specific formula of a Vieta flip comes in when we give an equational form to the $N_k$ counts. Indeed, a solution $x\in V(f)$ is fixed by the $i$-th Vieta flip when $(\partial_if)(x)=0$ or $(\partial^2_if)(x)=0$, possibly both. So we may write
\begin{align}
N_k(f)=\# \Big\{x\in V(f): (\partial_if)(x)\cdot (\partial^2_if)(x)=0 \textrm{ for } i=n-k+1, \dots, n\Big\}. 
\end{align}
In particular, if $f(x_1,\dots, x_n)=0$ is a regular equation then 
\begin{align}
N_k(f)=\# \Big\{x\in V(f): (\partial_if)(x)=0 \textrm{ for } i=n-k+1, \dots, n\Big\}. 
\end{align}

The general approach to counting the number of vertices and the degree distribution in Vieta graphs will be to work out the $N_k$ counts, followed by an application of Lemma~\ref{lem: Nk}.

\section{Quadratic character sums} 
The quadratic character $\sigma: \Fp\to \C$ is defined as follows: $\sigma(a)=1$ if $a\in (\Fp^*)^2$; $\sigma(a)=-1$ if $a\not\in (\Fp^*)^2$; and $\sigma(a)=0$ if $a=0$. The notation, $\sigma$, while not exactly standard, is meant to be suggestive: $\sigma$ is a signature-like map on squares. It is called a character on account of its multiplicativity: $\sigma(ab)=\sigma(a)\sigma(b)$ for all $a,b\in \Fp$. 

The quadratic character is useful for counting solutions to quadratic equations: the equation $x^2=c$ has $1+\sigma(c)$ solutions $x\in \Fp$; in general, for $a\neq 0$, we have
\begin{align}\label{eq: qen}
\#\big\{x \in \Fp: ax^2+bx+c=0\big\}=1+\sigma(b^2-4ac).
\end{align}

Several of our computations involve quadratic character sums. A complete, univariate quadratic character sum takes the form
\begin{align}\label{eq: chrsum}
\sum_{x\in \Fp} \sigma(P(x)) 
\end{align}
where $P(x)$ is a polynomial. Multivariate generalizations of \eqref{eq: chrsum} are easy to imagine, and in fact will show up frequently herein. If $P(x)$ is linear or quadratic, then there are well-known formulas for evaluating \eqref{eq: chrsum}: we have
\begin{align}\label{eq: qjs1}
\sum_{x\in \Fp} \sigma(bx+c) = 0 \qquad (b\neq 0),
\end{align}
and
\begin{align}\label{eq: qjs1.5}
\sum_{x\in \Fp} \sigma(ax^2+bx+c) = \begin{cases} -\sigma(a) & \textrm{ if } b^2-4ac\neq 0,\\
\sigma(a)\cdot (p-1) & \textrm{ if } b^2-4ac=0.
 \end{cases} \qquad (a\neq 0)
\end{align}
See, for instance, \cite[Thm.1.10]{N} and \cite[Thm.2.4]{N}. 

When $P(x)$ is cubic, explicit computations are quite rare. Standout instances when quadratic character sums with cubic arguments can be computed, are offered by \emph{Jacobsthal sums}.  We return to this topic in Section~\ref{sec: Jac}. For a systematic study, we refer to the recent book \cite{N}. 

A very useful tool in carrying on computations that involve cases, such as the formula \eqref{eq: qjs1.5}, is the Iverson bracket notation. This works as follows: if $\mathcal{S}$ is a statement, then $\llbracket \mathcal{S}\rrbracket=1$ if $\mathcal{S}$ is true, respectively $\llbracket \mathcal{S}\rrbracket=0$ if $\mathcal{S}$ is false. Using this notation, \eqref{eq: qjs1.5} can be expressed as follows: for $a\neq 0$,
\begin{align}\label{eq: qjs2}
\sum_{x\in \Fp} \sigma(ax^2+bx+c) =\sigma(a)\Big(\llbracket b^2=4ac\rrbracket \cdot p-1\Big).
\end{align}

Let us record slightly more technical versions of \eqref{eq: qen} and \eqref{eq: qjs2}, which include the case $a=0$.

\begin{lem} Let $f(x)=ax^2+bx+c$, and put $\Delta=b^2-4ac$. Then:
\begin{align}\label{eq: qen0}
N_0(f)=1+\sigma(\Delta)-\llbracket a=0\rrbracket+\llbracket a=b=c=0\rrbracket \cdot p
\end{align}
and
\begin{align}\label{eq: qjs0}
\sum_{x\in \Fp}\sigma(f(x))=-\sigma(a)+\sigma(a) \llbracket \Delta=0\rrbracket \cdot p+\sigma(c) \llbracket a=b=0\rrbracket \cdot p.
\end{align}
\end{lem}

\begin{proof}
We check formula \eqref{eq: qen0}. When $a\neq 0$, this is just \eqref{eq: qen}. When $a=0$, \eqref{eq: qen0} simplifies to 
\[\#\{x\in \Fp: bx+c=0\}=\sigma(b^2)+\llbracket b=c=0\rrbracket \cdot p.\] 
For $b\neq 0$, both sides equal $1$. For $b=0$, both sides equal $\llbracket c=0\rrbracket \cdot p$.

Next, we check formula \eqref{eq: qjs0}. The case $a\neq 0$ is \eqref{eq: qjs2}. For $a=0$, \eqref{eq: qjs0} simplifies to 
\[\sum_{x\in \Fp}\sigma(bx+c)=\sigma(c)\llbracket b=0\rrbracket\cdot p.\] 
When $b\neq 0$, this is just \eqref{eq: qjs1}. When $b=0$, both sides equal $\sigma(c)\cdot p$.
\end{proof}

For the matter at hand--the evaluation of the $N_k$ counts--multivariate quadratic character sums come up in the $N_0$ count and the $N_1$ count. The general lemma below makes this precise. In low dimensions, the counts for $N_2$ and higher are usually done more directly. 

\begin{disclem}\label{lem: start!}
Let $f(x_1,\dots,x_n)$ be a symmetric polynomial with coefficients in $\Fp$, having degree $2$ in each variable. Write 
\[f(x_1,\dots,x_n)=\alpha_nx_n^2+\beta_nx_n+\gamma_n,\] 
where $\alpha_n,\beta_n,\gamma_n$ are symmetric polynomials in the remaining variables $x_1,\dots,x_{n-1}$, and put
\begin{align}\label{eq: Deltagen}
\Delta(x_1,\dots,x_{n-1})=\beta_n^2-4\alpha_n\gamma_n.
\end{align}

Then the following hold:
\begin{align}\label{eq:N0gen}
N_0(f)&=p^{n-1}-N_0(\alpha_n)+pN_0(\alpha_n, \beta_n, \gamma_n)+\sum_{x\in \Fp^{n-1}} \sigma(\Delta(x)),
\end{align}
\begin{align}\label{eq:N1gen}
N_1(f)&=N_0(\Delta)+N_0(\alpha_n)-2N_0(\alpha_n,\beta_n)+pN_0(\alpha_n, \beta_n, \gamma_n).
\end{align}
Also
\begin{align}\label{eq:Qgen}
\begin{split}
\sum_{x\in \Fp^n} \sigma(f(x))=&-\sum_{x\in \Fp^{n-1}} \sigma(\alpha_n(x))+p\sum_{x\in \Fp^{n-1}:\: \Delta(x)=0} \sigma(\alpha_n(x))\\
&+p\sum_{x\in \Fp^{n-1}:\: \alpha_n(x)=\beta_n(x)=0} \sigma(\gamma_n(x)).
\end{split}
\end{align}
\end{disclem}

\begin{proof} Recall that $N_0(f)$ counts the number of solutions $(x_1,\dots,x_n)\in \Fp^n$ to the equation $f(x_1,\dots,x_n)=0$. Fix $x=(x_1,\dots,x_{n-1})\in \Fp^{n-1}$. In view of \eqref{eq:N0gen}, the equation $\alpha_n(x)x_n^2+\beta_n(x)x_n+\gamma_n(x)=0$ has
\[1+\sigma(\Delta(x))-\llbracket \alpha_n(x)=0\rrbracket+\llbracket \alpha_n(x)=\beta_n(x)=\gamma_n(x)=0\rrbracket \cdot p\]
solutions $x_n\in \Fp$. Summing over $x\in \Fp^{n-1}$, we obtain \eqref{eq:N0gen}.

Next, $N_1(f)$ counts the number of solutions $(x_1,\dots,x_n)\in \Fp^n$ to the equation $f(x_1,\dots,x_n)=0$ that satisfy $\alpha_n(x_1,\dots,x_{n-1})=0$, or $\alpha_n(x_1,\dots,x_{n-1})\neq 0$ and $2x_n=-\beta_n(x_1,\dots,x_{n-1})/\alpha_n(x_1,\dots,x_{n-1})$. Fix $x=(x_1,\dots,x_{n-1})\in \Fp^{n-1}$. When $\alpha_n(x)=0$, there are $\llbracket \beta_n(x)\neq 0\rrbracket+\llbracket \beta_n(x)=\gamma_n(x)=0\rrbracket\cdot p$ solutions $x_n\in \Fp$ to the equation $f(x_1,\dots,x_n)=0$. When $\alpha_n(x)\neq 0$, there are $\llbracket \Delta(x)=0\rrbracket$ solutions $x_n\in \Fp$ to the equation $f(x_1,\dots,x_n)=0$. Overall, $x\in \Fp^{n-1}$ contributes
\[\llbracket \alpha_n(x)=0\rrbracket \Big(\llbracket \beta_n(x)\neq 0\rrbracket+\llbracket \beta_n(x)=\gamma_n(x)=0\rrbracket\cdot p\Big)+\llbracket \Delta(x)=0\rrbracket\cdot \llbracket \alpha_n(x)\neq 0\rrbracket\]
to the $N_1$ count. Next, we replace $\llbracket \beta_n(x)\neq 0\rrbracket=1-\llbracket \beta_n(x)=0\rrbracket$ and $\llbracket \alpha_n(x)\neq 0\rrbracket=1-\llbracket \alpha_n(x)=0\rrbracket$. We note that
\begin{align*}
\llbracket \Delta(x)=0\rrbracket\cdot \llbracket \alpha_n(x)\neq 0\rrbracket&=\llbracket \Delta(x)=0\rrbracket\Big(1- \llbracket \alpha_n(x)= 0\rrbracket\Big)\\
&=\llbracket \Delta(x)=0\rrbracket-\llbracket \alpha_n(x)=\Delta(x)= 0\rrbracket\\
&=\llbracket \Delta(x)=0\rrbracket-\llbracket \alpha_n(x)=\beta_n(x)= 0\rrbracket.
\end{align*}
Thus, the contribution of each $x\in \Fp^{n-1}$ is
\begin{align*}
\llbracket \Delta(x)=0\rrbracket&+\llbracket \alpha_n(x)=0\rrbracket -2\llbracket \alpha_n(x)=\beta_n(x)=0\rrbracket\\
&+p\llbracket \alpha_n(x)=\beta_n(x)=\gamma_n(x)=0\rrbracket.
\end{align*}
Summing over $x\in \Fp^{n-1}$, we obtain \eqref{eq:N1gen}.

Finally, fix $x=(x_1,\dots,x_{n-1})\in \Fp^{n-1}$. By \eqref{eq: qjs0}, we have that
\begin{align*}
\sum_{x_n\in \Fp} \sigma\big(\alpha_n(x)x_n^2+\beta_n(x)x_n+\gamma_n(x)\big)
\end{align*}
equals
\begin{align*}
-\sigma(\alpha_n(x))+\sigma(\alpha_n(x)) \llbracket \Delta(x)=0\rrbracket \cdot p+\sigma(\gamma_n(x)) \llbracket \alpha_n(x)=\beta_n(x)=0\rrbracket \cdot p.
\end{align*}
Summing over $x\in \Fp^{n-1}$, we obtain \eqref{eq:Qgen}.
\end{proof}

The point of the Discriminant Lemma is the dimension drop: in each one of \eqref{eq:N0gen}, \eqref{eq:N1gen}, and \eqref{eq:Qgen}, the left-hand side is $n$-dimensional whereas the right-hand side is $(n-1)$-dimensional. The relevance of \eqref{eq:N0gen} and \eqref{eq:N1gen} is clear, but that of \eqref{eq:Qgen} is less so. Note that \eqref{eq:N0gen} involves a multidimensional character sum with argument $\Delta(x_1,\dots,x_{n-1})$. If the latter polynomial happens to have degree $2$ in each variable, then we can apply \eqref{eq:Qgen} to $\Delta(x_1,\dots,x_{n-1})$, and we can apply \eqref{eq:N0gen} to compute the $N_0(\Delta)$ term in \eqref{eq:N1gen}.

In several applications of the above Lemma, $\alpha_n$ will be constant. Therefore, the terms which include the zero-set of $\alpha_n$ will vanish; formulas \eqref{eq:N0gen}, \eqref{eq:N1gen}, \eqref{eq:Qgen} will be pleasingly simpler.

\bigskip
\part{Vieta graphs in dimension $3$} 
\section{The generalized Markoff cubic}\label{sec: gm}
Consider the equation
\begin{align*}
\M(C,D): \qquad x^2+y^2+z^2=xyz-C(x+y+z)+D.
\end{align*}
The cubic $\M(C,D)$ may be thought of as a two-parameter deformation of the Markoff cubic \eqref{eq: M} which, in the present notation, is the `origin' $\M(0,0)$.  

The generalized Markoff equation $\M(C,D)$ we are interested in is, up to changing the sign of each variable, the symmetric case of the following equation:
\begin{align}\label{eq: GNS}
x^2+y^2+z^2+xyz=Ax+By+Cz+D.
\end{align}
The real or complex cubic surface defined by \eqref{eq: GNS} has received a lot of attention in the recent literature. It arises as the character variety of a sphere with four punctures \cite[Sec.5.2]{Gold}. The dynamics of the group action generated by the Vieta flips on the real or complex surface \eqref{eq: GNS} were studied in \cite{Can, CL, RR}, among others. In light of all this, the idea of a Markoff-type graph defined by the equation \eqref{eq: GNS} is very natural, and already hinted at in \cite{BGS}. Our viewpoint is that symmetry is not only appealing, but also an essential aspect of the Markov/Markoff landscape. We have therefore settled on $\M(C,D)$ as the equation underlying the `right' generalization of the Markoff graph over $\Fp$.

In fact $\M(C,D)$ accounts for the most general symmetric equation of Markoff type over $\Fp$. Such an equation takes the form
\[x^2+y^2+z^2=C_3xyz+ C_2(xy+yz+zx)+C_1(x+y+z)+C_0\]
where $C_3\neq 0$. However, by uniform scalings ($x:=x/C_3$, $y:=y/C_3$, $z:=z/C_3$) followed by uniform shifts ($x:=x-C_2$, $y:=y-C_2$, $z:=z-C_2$), we may arrange that $C_3=1$ and $C_2=0$. This normalization amounts to our generalized Markoff equation.

As a relevant example, let us normalize equation \eqref{eq: GSLM}. The uniform scalings $x:=x/(3+3K)$ etc., yield the nice equivalent form 
\begin{align}\label{eq: KMar0}
x^2+y^2+z^2+K(xy+yz+zx)=xyz.
\end{align}
Next, the uniform shifts $x:=x+K$ etc., lead to
\begin{align}\label{eq: KMar}
x^2+y^2+z^2=xyz-(K^2+2K)(x+y+z)-(2K^3+3K^2).
\end{align}
This is equation $\M(C,D)$ with $C=K^2+2K$ and $D=-(2K^3+3K^2)$. A minor difference is that the parameter value $K=-1$, which was excluded in \eqref{eq: GSLM}, is now allowed. In fact, it will turn out that the value $K=-1$ plays a distinguished role within the \emph{$K$ family} described by \eqref{eq: KMar}. 

At this point, we make a general remark, valid in any dimension: linear changes of variables--shifts $x_i:=x_i+c_i$, or scalings $x_i:=s_ix_i$ with $s_i\neq 0$--induce obvious graph isomorphisms between the corresponding Vieta graphs. This allows us to freely move between linearly equivalent forms of a polynomial equation $f(x_1,\dots,x_n)=0$, driven by normalizations or the occasional pretty presentation. Incidentally, these two aspects can be quite different: defining the $K$ family by \eqref{eq: KMar0} is aesthetically pleasing; however, the normalization \eqref{eq: KMar} is functionally more convenient for our purposes. By viewing the $K$ family as a generalized Markoff equation of the form $\M(C,D)$ we are gaining not just a more organized conceptual framework, but also a practical benefit--namely, simplified Vieta flips.  

The Vieta graph of $\M(C,D)$ has the solutions $(x,y,z)\in \Fp^3$ to equation $\M(C,D)$ as its vertices; edges connect a solution $(x,y,z)$ to its Vieta flips 
\[\big(yz-C-x,y,z\big), \; \big(x,xz-C-y,z\big), \; \big(x,y,xy-C-z\big),\] 
which are still solutions to equation $\M(C,D)$. While equation $\M(C,D)$ and its corresponding Vieta graph depend on two parameters, it seems reasonable to interpret $C$ as the principal parameter; it is the only one involved in the Vieta flips. The secondary parameter $D$ is a `level' parameter. 

We have already alluded to the point that the search for explicit counting results helps us discern interesting instances of the generalized Markoff equation $\M(C,D)$. A stand-out one-parameter family is the \emph{Cayley family}, wherein for each $C$ we pick the level
\begin{align}\label{eq: D}
D_{\bullet}=4-2C-\frac{C^2}{4}.
\end{align}
An exceptional member of this family is the Cayley cubic $\M(0,4)$; some of this cubic's features are present throughout the family, whence the name. Our analysis will highlight another remarkable member of the Cayley family--the cubic $\M(8,-28)$.

The form-versus-function point we made before comes up again. The normalization of the Cayley family is the not-so-appealing equation
\begin{align}\label{eq: C}
x^2+y^2+z^2=xyz-C(x+y+z)+4-2C-\frac{C^2}{4}.
\end{align}
However, the uniform shifts $x:= x-2$ etc. lead to the pleasant looking
\begin{align}\label{eq: C'}
\Big(x+y+z+\frac{C-8}{2}\Big)^2=xyz.
\end{align} 
Now, \eqref{eq: C'} calls for taking $C=8$, whereupon it becomes
\begin{align}\label{eq: sM}
(x+y+z)^2=xyz.
\end{align} 
Besides its intrinsic appeal, this equivalent form of the $\M(8,-28)$ cubic enjoys an interesting relation to the Markoff cubic. Namely, the squaring map $(x,y,z)\mapsto (x^2,y^2,z^2)$ sends solutions of $x^2+y^2+z^2=xyz$ to solutions of \eqref{eq: sM}.

The $K$ family, discussed above, turns out to be another one-parameter family that is distinguished from a structural viewpoint; see Section~\ref{sec: K}. The parameterization $C=K^2+2K$ and $D=-(2K^3+3K^2)$ includes, in particular, the Markoff cubic $\M(0,0)$ for $K=0$; the Cayley cubic $\M(0,4)$ for $K=-2$; and the remarkable cubic $\M(8,-28)$ for $K=2$. 

We will get explicit degree counts for the entire Cayley family and the entire $K$ family. For other one-parameter families of interest, our goal of obtaining explicit degree counts will be achieved only partially, on certain members of the families. 

Consider the \emph{Fricke family}, defined by $C=0$. The equation $\M(0,D)$ is, of course, exactly \eqref{eq: G0}. In addition to the Markoff cubic $\M(0,0)$ and the Cayley cubic $\M(0,4)$, we will get explicit degree counts for the Fricke cubics $\M(0,-4)$, $\M(0,2)$, and $\M(0,8)$.

A more subtle choice is the \emph{$J$ family}, parameterized by $C=3J^2-4$, $D=9J^3-36J-28$. We will be able to get explicit degree counts for a subfamily of the $J$ family, which includes the more exotic cubic $\M(-8/3,-20/3)$.

The explicit degree counts within the Fricke family and the $J$ family draw from certain quadratic character sums known as Jacobsthal sums. We summarize the essentials of this topic in Section~\ref{sec: Jac}.

\section{The $N_0$ count}
We begin the $N_k$ counts for the polynomial
\begin{align}\label{eq: mainf}
f(x,y,z)=x^2+y^2+z^2-xyz+C(x+y+z)-D.
\end{align}
In this section, we handle the $N_0$ count. The counts for $N_1$, $N_2$, and $N_3$ are the subject of Sections~\ref{sec: N1}, ~\ref{sec: N2}, and~\ref{sec: N3} respectively.

\begin{thm}\label{thm: v}
The following hold:
\begin{itemize}
\item[(i)] if $D=D_\bullet$, then 
\begin{align}\label{eq: N0countC}
N_0(f)=p^2+1+\sigma(C^2-8C) \cdot p;
\end{align}
\item[(ii)] if $D\neq D_\bullet$, then 
\begin{align}\label{eq: N0count}
N_0(f)=p^2+1+\sigma(D-D_\bullet)\big(1+N_0(\kappa)\big) \cdot p
\end{align}
where $N_0(\kappa)$ is the number of distinct roots in $\Fp$ of the cubic 
\begin{align}\label{eq: kappa}
\kappa(x)=x^3+(C-2)x^2-(2C+D)x+C^2+2D.
\end{align}
\end{itemize}
\end{thm}

We note, in particular, that $N_0(f)=p^2+1+ mp$ where $m$ is a small integer--it satisfies $|m|\leq 4$. More precisely $m\in \{0,\pm 1\}$ in case (i), and $m\in \{\pm 1, \pm 2, \pm 3, \pm 4\}$ in case (ii).

\begin{proof}
We apply \eqref{eq:N0gen} to $f(x,y,z)$. Here 
\begin{align*}
\alpha_z(x,y)=1,\quad \beta_z(x,y)=-(xy-C), \quad \gamma_z(x,y)=x^2+y^2+C(x+y)-D,
\end{align*}
and
\begin{align}\label{eq: Delta}
\begin{split}
\Delta(x,y)&=\beta_z(x,y)^2-4\alpha_z(x,y)\gamma_z(x,y)\\
&=(x^2-4)y^2-2C(x+2)y-4x^2-4Cx+C^2+4D.
\end{split}
\end{align}

We get
\begin{align}\label{eq: N0part}
N_0(f)=p^2+\sum_{x,y\in \Fp}\sigma(\Delta(x,y)).
\end{align}

Next, we compute the latter quadratic character sum by applying \eqref{eq:Qgen} to $\Delta(x,y)$. We have
\begin{align*}
\alpha(x)=x^2-4,\qquad \beta(x)=-2C(x+2), \qquad \gamma(x)=-4x^2-4Cx+C^2+4D,
\end{align*}
and
\begin{align}
\Delta(x)&=\beta(x)^2-4\alpha(x)\gamma(x)=16(x+2)\:\kappa(x)
\end{align}
where $\kappa$ is the cubic in \eqref{eq: kappa}. Thus, the bivariate character sum in \eqref{eq: N0part} equals
\begin{align}\label{eq: N0party}
-\sum_{x\in \Fp} \sigma(\alpha(x))+p\sum_{x\in \Fp:\: \Delta(x)=0} \sigma(\alpha(x))+p\sum_{x\in \Fp:\: \alpha(x)=\beta(x)=0} \sigma(\gamma(x)).
\end{align}
The first, complete sum is the simplest among these three univariate character sums: we clearly have
\[\sum_{x\in \Fp} \sigma(\alpha(x))=-1.\] 
Between the two remaining short sums, the second is easier. Note that $\alpha(x)=\beta(x)=0$ has one solution $x=-2$ when $C\neq 0$, respectively two solutions $x=\pm 2$ when $C=0$. We have $\gamma(-2)=-16+8C+C^2+4D=4(D-D_\bullet)$; also $\gamma(2)=4(D-D_\bullet)$ when $C=0$. Therefore
\begin{align*}
\sum_{x\in \Fp:\: \alpha(x)=\beta(x)=0} \sigma(\gamma(x))&=\sigma(\gamma(-2))+\llbracket C=0\rrbracket\sigma(\gamma(2))\\
&=\sigma(D-D_\bullet)\big(1+\llbracket C=0\rrbracket\big).
\end{align*}
Note, in particular, that this short sum vanishes when $D=D_\bullet$.

We turn to the evaluation of the first short sum. We can write
\begin{align*}
\sum_{x\in \Fp:\: \Delta(x)=0} \sigma(\alpha(x))=\sum_{x\in \Fp:\: \kappa(x)=0} \sigma(x^2-4).
\end{align*}
Key to what follows is the following alternate expression for the cubic $\kappa$:
\begin{align}\label{eq: kappa2}
\kappa(x)=(x+2)\left(x-2+\frac{C}{2}\right)^2-(D-D_\bullet)(x-2).
\end{align}

(i) Assume $D=D_\bullet$. Then $\kappa(x)=(x+2)(x-2+C/2)^2$, with roots $-2$ and $2-C/2$. In this case 
\[\sum_{x\in \Fp:\: \kappa(x)=0} \sigma(x^2-4)=\sigma\big((2-C/2)^2-4\big)=\sigma(C^2-8C).\]
Now \eqref{eq: N0party} gives
\begin{align*}
\sum_{x,y\in \Fp}\sigma(\Delta(x,y))=1+\sigma(C^2-8C)\cdot p
\end{align*}
and the $N_0$ count claimed in \eqref{eq: N0countC} follows.

(ii) Assume $D\neq D_\bullet$. Consider first the case $C\neq 0$. If $x_0$ is a root of $\kappa(x)$, then $(x_0+2)(x_0-2+C/2)^2=(D-D_\bullet)(x_0-2)$. Necessarily $x_0\neq 2$ and $x_0\neq 2-C/2$. We infer that $\sigma(x_0+2)=\sigma(D-D_\bullet)\cdot \sigma(x_0-2)$, so $\sigma(x_0^2-4)=\sigma(D-D_\bullet)$ whenever $x_0$ is a root of $\kappa(x)$. Therefore
\[\sum_{x\in \Fp:\: \kappa(x)=0} \sigma(x^2-4)=\sigma(D-D_\bullet)N_0(\kappa) \qquad \textrm{ if } C\neq 0.\]
Next, consider the case $C=0$. Then $\kappa(x)=(x-2)\big(x^2-4-(D-D_\bullet)\big)$. A root $x_0$ of the quadratic factor $x^2-4-(D-D_\bullet)$ satisfies $x_0\neq 2$, so there are $N_0(\kappa)-1$ such roots. Furthermore, $\sigma(x_0^2-4)=\sigma(D-D_\bullet)$ continues to hold. Therefore
\[\sum_{x\in \Fp:\: \kappa(x)=0} \sigma(x^2-4)=\sigma(D-D_\bullet)\big(N_0(\kappa) -1\big)\qquad \textrm{ if } C=0.\]
Summarizing the evaluations involved in \eqref{eq: N0party}, we see that
\begin{align*}
\sum_{x,y\in \Fp}\sigma(\Delta(x,y))=1+\sigma(D-D_\bullet)\big(1+N_0(\kappa)\big) \cdot p
\end{align*}
and the $N_0$ count claimed in \eqref{eq: N0count} follows.
\end{proof}

In \cite{Car3} Carlitz has computed the solution count for the equation 
\begin{align}\label{eq: abcd}
a_1x^2+a_2y^2+a_3z^2=2bxyz+2c_1x+2c_2y+2c_3z+d
\end{align}
where $a_1a_2a_3b\neq 0$. This is, essentially, equation \eqref{eq: GNS} when $a_1=a_2=a_3$ and, more generally, when $\sigma(a_1)=\sigma(a_2)=\sigma(a_3)$. Carlitz's result is evidently more general, but also more complicated on two accounts. Firstly, he has to use a quartic polynomial \cite[eq.(12)]{Car3} in place of our cubic $\kappa$. Secondly, he cannot compute--explicitly, as we did--a term \cite[eq.(18)]{Car3} that corresponds to what we called the first short sum in the above proof. Our explicit computation has been greatly facilitated by the symmetry of $\M(C,D)$. 

Carlitz has obtained solution counting results for closely related equations in other papers as well. The particular case $c_1=c_2=c_3$ of \eqref{eq: abcd} was first dealt with in \cite{Car1}. This covers, in particular, the Fricke cubic $x^2+y^2+z^2=xyz+ D$. The solution count for the Fricke cubic is commonly encountered in the literature; see \cite[Lem.6.4]{GS}, \cite[Prop.2.1]{CIL}, \cite[Prop.2.2]{Cou}, \cite[Ex.2.7]{N}.

In \cite{Car2}, Carlitz has computed the solution count for the symmetric equation $(x+y+z-b)^2=2axyz$. This, too, is a particular case of \eqref{eq: abcd}. In fact, in view of \eqref{eq: C'}, we recognize that this solution count is exactly that for $\M(C,D)$ in what we termed the Cayley case, $D=D_\bullet$.

 To summarize, much of Theorem~\ref{thm: v} could be ascribed to Carlitz. Our contribution is to go a step further in part (ii), eventually obtaining a much more explicit counting result.
 
\begin{rem} We point out an interesting coincidence, for which we do not currently have an explanation. A distinguished set of parameters in the study of the real or complex surface \eqref{eq: GNS}, originating from considerations of mathematical physics, is the so-called Dubrovin--Mazzocco parameters; see \cite{RR}. For our cubic $\M(C,D)$, they correspond to real parameters $C=2(a+2)$, $D=-(a^2+8a+8)$ for $-2<a<2$. We observe that $D=D_{\bullet}$, so the Dubrovin--Mazzocco parameters belong to the Cayley family. Furthermore, and most intriguingly, the real condition $-2<a<2$ can be rephrased as $C(C-8)<0$; the finite field analogue of the latter is $\sigma(C^2-8C)=-1$. Intriguingly, $\sigma(C^2-8C)$ is precisely the quadratic signature which appears in the $N_0$ count \eqref{eq: N0countC}.
\end{rem} 

\section{The $N_1$ count}\label{sec: N1} 
\begin{thm}\label{thm: N1} 
For the polynomial $f$ of \eqref{eq: mainf}, the following hold:
\begin{itemize}
\item[(i)] if $D=D_\bullet$, then
\begin{align}\label{eq: N1countC}
N_1(f)=\big(3+\llbracket C=0\rrbracket \big) p-4+\llbracket C=8\rrbracket.
\end{align}
\item[(ii)] if $D\neq D_\bullet$, then
\begin{align}\label{eq: N1count}
N_1(f)=p-3+\sum_{x\in \Fp}\sigma(\lambda(x))
\end{align}
where $\lambda(x)$ is the cubic
\begin{align}\label{eq: lambda}
\lambda(x)=x^3-\big(D-D_\bullet\big)x^2+\Big(\frac{C-8}{2}\cdot x+4(D-D_\bullet)\Big)^2.
\end{align}
\end{itemize}
\end{thm}

\begin{proof} We keep the notations from the proof of Theorem~\ref{thm: v}. Applying \eqref{eq:N1gen} to $f(x,y,z)$, and keeping in mind that $\alpha_z(x,y)$ never vanishes, we simply get $N_1(f)=N_0(\Delta)$. We compute the latter by applying \eqref{eq:N0gen} to $\Delta(x,y)$. We deduce that
\[N_1(f)=p-N_0(\alpha)+pN_0(\alpha, \beta, \gamma)+\sum_{x\in \Fp}\sigma(\Delta(x)).\]
Clearly, $N_0(\alpha)=2$. Moving on to $N_0(\alpha, \beta, \gamma)$, we distinguish two cases. If $C\neq 0$, then the only common root of $\alpha$ and $\beta$ is $x=-2$; this is a root of $\gamma$ if and only if $D=D_\bullet$. If $C=0$, then the common roots of $\alpha$ and $\beta$ are $x=\pm 2$; these are roots of $\gamma$ if and only if $D=D_\bullet$, once again. Thus
\[N_0(\alpha, \beta, \gamma)=\big(1+\llbracket C=0\rrbracket \big)\llbracket D=D_\bullet\rrbracket.\]
In order to address the univariate character sum, we split the discussion into the easy Cayley case, and the non-Cayley case.

(i) Assume $D=D_\bullet$. Then $\Delta(x)=16(x+2)\kappa(x)=16(x+2)^2(x-2+C/2)^2$, so
\begin{align*}
\sum_{x\in \Fp}\sigma(\Delta(x))&=p-\#\{x\in \Fp: \Delta(x)=0\}=p-\big(2-\llbracket C=8\rrbracket \big).
\end{align*}
Combining our computations, we obtain 
\[N_1(f)=p-2+\big(1+\llbracket C=0\rrbracket \big)p+p-2+\llbracket C=8\rrbracket.\]
This verifies \eqref{eq: N1countC}.

(ii) Assume $D\neq D_\bullet$. As $\Delta(x)=16(x+2)\kappa(x)$, we can write
\begin{align*}
\sum_{x\in \Fp}\sigma(\Delta(x))=\sum_{x\in \Fp}\sigma\big((x+2)\kappa(x)\big)=\sum_{x\in \Fp}\sigma\big(x\kappa(x-2)\big).
\end{align*}
We compute $x\kappa(x-2)=x^4+(C-8)x^3-(6C+D-20)x^2+4(D-D_\bullet)x$.
Next, using the quartic-to-cubic transformation lemma quoted below, we get 
\[\sum_{x\in \Fp}\sigma(\Delta(x))=-1+\sum_{x\in \Fp} \sigma(\lambda(x))\]
where $\lambda(x)$ is the cubic
\begin{align*}
\lambda(x)=x^3-\big(6C+D-20\big)x^2+4(C-8)(D-D_\bullet)x+16(D-D_\bullet)^2.
\end{align*}
A short calculation reveals that this is precisely the cubic spelled out in \eqref{eq: lambda}. In this non-Cayley case, we conclude that
\[N_1(f)=p-3+\sum_{x\in \Fp} \sigma(\lambda(x)),\]
as claimed.
\end{proof}

\begin{lem}\label{lem: q2c}
Let $a_1, a_2, a_3\in \Fp$ with $a_1\neq 0$. Then
\begin{align*}
\sum_{x\in \Fp} \sigma\big(x^4+a_3 x^3+a_2 x^2+a_1x\big)=-1+\sum_{x\in \Fp} \sigma\big(x^3+a_2 x^2+a_1a_3 x+a_1^2\big).
\end{align*}
\end{lem}

\begin{proof}
Put $P(x)=x^4+a_3 x^3+a_2 x^2+a_1x$. Then
\begin{align*}
\sum_{x\in \Fp} \sigma(P(x))=\sum_{x\in \Fp^*} \sigma(P(x))=\sum_{x\in \Fp^*} \sigma(P(x^{-1}))=\sum_{x\in \Fp^*} \sigma\big(x^4P(x^{-1})\big).
\end{align*}
Now $x^4P(x^{-1})=1+a_3x+a_2x^2+a_1x^3$, so
\begin{align*}
\sum_{x\in \Fp} \sigma(P(x))&=\sum_{x\in \Fp^*} \sigma\big(a_1x^3+a_2x^2+a_3x+1\big)\\
&=-1+\sum_{x\in \Fp} \sigma\big(a_1x^3+a_2x^2+a_3x+1\big).
\end{align*}
In the latter sum, we make the change of variable $x:=x/a_1$ and we clear the denominator. We  obtain
\[
\sum_{x\in \Fp} \sigma(P(x))=-1+\sum_{x\in \Fp} \sigma\big(x^3+a_2 x^2+a_1a_3 x+a_1^2\big),
\]
as desired.
\end{proof}

In general, one expects the Vieta graph of an equation to be non-regular as soon as $p$ is large enough. As an application of Theorem~\ref{thm: N1}, we determine the finitely many exceptions in the case of $\M(C,D)$.

\begin{cor}
The Vieta graph of $\M(C,D)$ is not regular, except for $\M(0,2)$ over $\F_{\!5}$ and $\M(0,5)$ over $\F_{\!7}$.
\end{cor}

\begin{figure}[ht]
    \centering
    \includegraphics[width=1.0\textwidth]{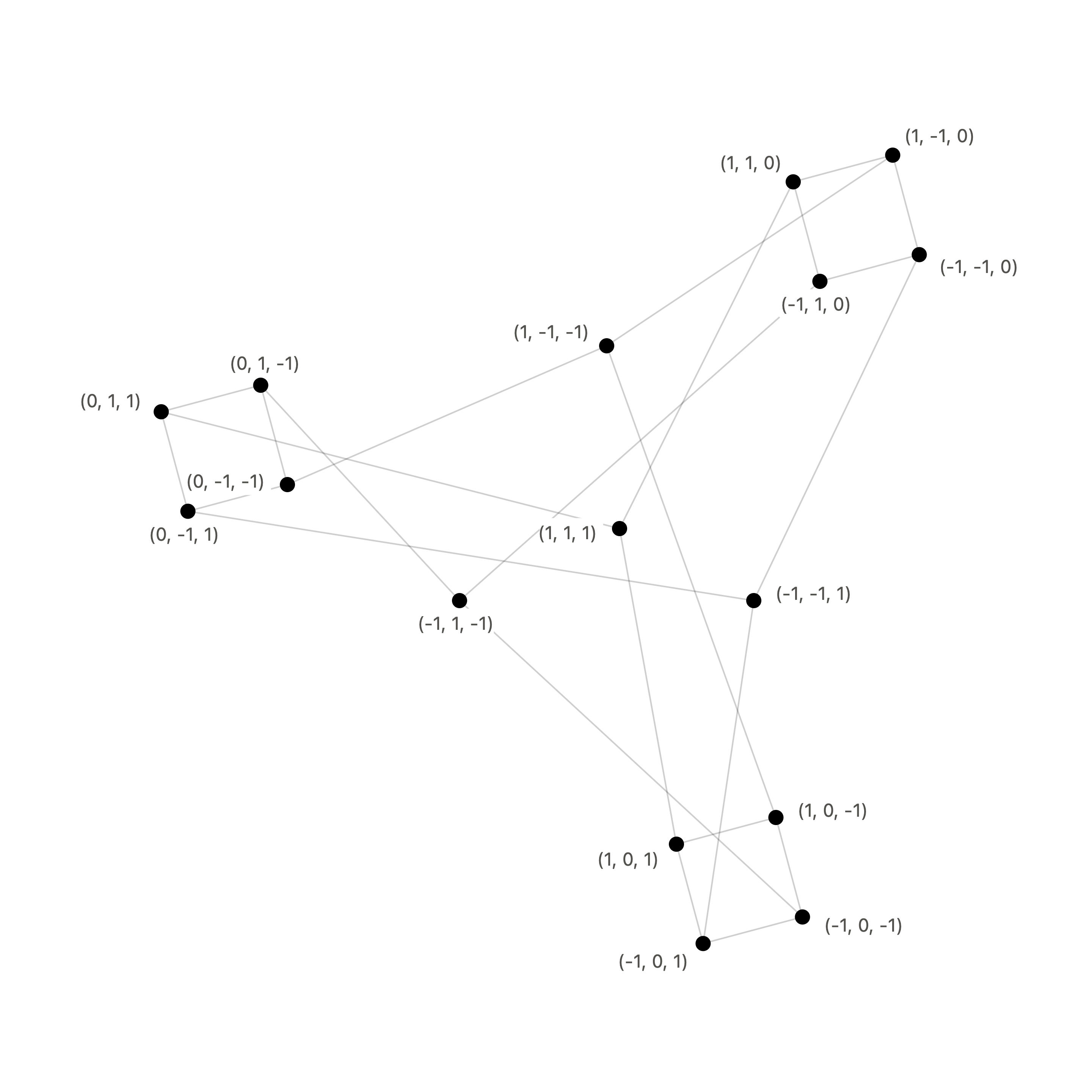}
    \vspace{-1cm}\caption{The Vieta graph of the cubic $\M(0,2)$ over $\F_{\!5}$, a connected $3$-regular graph on $16$ vertices.}
    \label{fig: C0_D2}
\end{figure}

\begin{proof}
By Corollary~\ref{cor: reg}, the Vieta graph of $\M(C,D)$ is regular if and only if $N_1(f)=0$. If $D=D_\bullet$ then $N_1(f)\geq 3p-4$. Thus, $N_1(f)=0$ if and only if $D\neq D_\bullet$ and 
\begin{align}\label{eq: smalllambda}
\sum_{x\in \Fp} \sigma(\lambda(x))=-(p-3).
\end{align}
We will determine that \eqref{eq: smalllambda} can only hold in very exceptional circumstances.

The first observation is that $p$ has to be very small. Indeed, the Hasse--Weil bound
\begin{align}\label{eq: HW}
\Big|\sum_{x\in \Fp}\sigma(\lambda(x))\Big|\leq 2\sqrt{p}
\end{align}
implies that $p-3\leq 2\sqrt{p}$; equivalently, $p\leq 9$. Thus $p=5$ or $p=7$.

Next we put $d:=D-D_\bullet$, so $d\neq 0$, and we write
\begin{align*}
\lambda(x)=x^3-dx^2+\Big(\frac{C-8}{2}\cdot x+4d\Big)^2.
\end{align*}
The second observation is that $\lambda(0)=(4d)^2$ and $\lambda(d)=(Cd/2)^2$, whence $\sigma(\lambda(0))=1$ and $\sigma(\lambda(d))=\sigma(C^2)$. If $C\neq 0$, then $\sigma(C^2)=1$ and \eqref{eq: smalllambda} becomes 
\[\sum_{x\neq 0,d} \sigma(\lambda(x))=-(p-1).\]
This, however, is impossible since the lowest value that the left-hand sum can achieve is $-(p-2)$. The upshot is that $C=0$. 

Now $\lambda(x)=x^3-dx^2+16(x-d)^2$. The change of variable $x:=dx$ gives
\[\sum_{x\in \Fp} \sigma(\lambda(x))=\sum_{x\in \Fp}\sigma\big(d(x^3-x^2)+16(x-1)^2\big),\]
and \eqref{eq: smalllambda} leads to
\[\sum_{x\neq 0,1}\sigma\big(d(x^3-x^2)+16(x-1)^2\big)=-(p-2).\]
This means that the cubic $d(x^3-x^2)+16(x-1)^2$ evaluates to a non-square in $\Fp^*$ whenever $x\neq 0,1$. When $p=5$ we find $d=3$, and when $p=7$ we find $d=1$. Recalling that $d=D-D_\bullet$ and $D_\bullet=4$, as $C=0$, this means that $D=2$ when $p=5$, respectively $D=5$ when $p=7$. \end{proof}

\section{The $N_2$ count}\label{sec: N2} 
The counts for $N_2$ and $N_3$, obtained in this section respectively in Section~\ref{sec: N3}, are direct. 

\begin{thm}\label{thm: N2} For the polynomial $f$ of \eqref{eq: mainf}, the following hold:
\begin{itemize}
\item[(i)] if $D=D_\bullet$, then
\[N_2(f)=\big(1+\llbracket C=0\rrbracket\big)p+1-\llbracket C=0\rrbracket-\llbracket C=8\rrbracket;\]
\item[(ii)] if $D\neq D_\bullet$, then
\[N_2(f)=N_0(\kappa)-\llbracket C=0\rrbracket\]
\end{itemize}
where $\kappa(x)$ is the cubic in \eqref{eq: kappa}.
\end{thm}

\begin{proof}
We recall that $N_2(f)$ counts the number of solutions $(x,y,z)$ to the system $f=0$, $\partial_y f=0$, $\partial_z f=0$; in other words, the solutions to $f(x,y,z)=0$ that satisfy $2z=xy-C$ and $2y=xz-C$. We have $z=(xy-C)/2$, and our system reduces to
\[
\begin{cases}
(x^2-4)y=C(x+2),\\
\Delta(x,y)=0,
\end{cases}
\]
where $\Delta(x,y)$ is the discriminant given by \eqref{eq: Delta}. 

Consider solutions with $x=-2$, in which case the system reduces to $\Delta(-2,y)=0$. But $\Delta(-2,y)=-16+8c+C^2+4D=4(D-D_\bullet)$. Thus, there are no solutions when $D\neq D_\bullet$, and $p$ solutions of the form $(-2,y, -y-C/2)$ when $D=D_\bullet$.

Consider solutions with $x=2$. The system becomes $4C=0$, $\Delta(2,y)=0$. There are no solutions when $C\neq 0$. When $C=0$, we have $D_\bullet=4$ and $\Delta(2,y)=-16+4D=4(D-D_\bullet)$. So there are no solutions when $D\neq D_\bullet$, and $\llbracket C=0\rrbracket p$ solutions of the form $(2,y,y-C/2)$ when $D=D_\bullet$.

Finally, consider solutions with $x\neq \pm 2$. Then $y=C/(x-2)$, and $z=C/(x-2)$ as well. The companion equation $\Delta(x,y)=0$ turns into $\kappa(x)=0$. We thus have solutions of the form $(x_0,C/(x_0-2), C/(x_0-2))$ where $x_0\neq \pm 2$ is a root of the cubic $\kappa(x)$. We now distinguish two subcases, and we refer to the root analysis for $\kappa$ that we performed in the proof of Theorem~\ref{thm: v}. If $D=D_\bullet$, then $\kappa(x)=(x+2)(x-2+C/2)^2$. The root $x_0=2-C/2$ is different from $\pm 2$ if and only if $C\neq 0,8$, in which case the solution is $(2-C/2, -2, -2)$. If $D\neq D_\bullet$, then the number of roots of the cubic $\kappa(x)$, which are different from $\pm 2$, is $N_0(\kappa)$ when $C\neq 0$, respectively $N_0(\kappa)-1$ when $C=0$. 

Let us summarize our findings.

(i) If $D=D_\bullet$, then there are $p$ solutions of the form $(-2,y,-y-C/2)$; $\llbracket C=0\rrbracket p$ solutions of the form $(2,y,y-C/2)$; and $1-\llbracket C=0\rrbracket-\llbracket C=8\rrbracket$ solutions of the form $(2-C/2, -2, -2)$. Overall, $N_2(f)=\big(1+\llbracket C=0\rrbracket\big)p+1-\llbracket C=0\rrbracket-\llbracket C=8\rrbracket$.

(ii) If $D\neq D_\bullet$, then there are $N_0(\kappa)-\llbracket C=0\rrbracket=N_2(f)$ solutions of the form $(x_0,C/(x_0-2), C/(x_0-2))$ with $x_0\neq \pm 2$. 
\end{proof}

In particular, we can give necessary and sufficient conditions for the absence of isolated and leaf vertices in our Vieta graphs.  

\begin{cor}\label{cor: no01}
The Vieta graph of $\M(C,D)$ has no vertices of degree $0$ or $1$ if and only if one of the following occurs:
\begin{itemize}
\item[(i)] $C=0$ and $\sigma(D)=-1$,
\item[(ii)] $C\neq 0$, $D\neq D_\bullet$, and the cubic $\kappa(x)$ has no roots in $\Fp$.
\end{itemize}
\end{cor}

\begin{proof} By Corollary~\ref{cor: reg}, a $3$-dimensional Vieta graph has no vertices of degree $0$ or $1$, if and only if $N_2=0$. Let us analyze how this arises by using Theorem~\ref{thm: N2}.

If $D=D_\bullet$ then $N_2(f)\geq p$. Thus, $N_2(f)=0$ if and only if $D\neq D_\bullet$ and $N_0(\kappa)=\llbracket C=0\rrbracket$. Let us make the latter condition more explicit. In the case $C\neq 0$, it amounts to $N_0(\kappa)=0$; we get (ii) as one possibility. In the case $C=0$, it amounts to $N_0(\kappa)=1$. Now $\kappa(x)=(x-2)(x^2-D)$, and $D\neq 4$ since $D_\bullet=4$. So $N_0(\kappa)=1$ if and only if $\sigma(D)=-1$, a condition which also subsumes the requirement that $D\neq 4$. We conclude that (i) is the other possibility.
\end{proof}

Our counts thus far depend on two univariate cubics: $\kappa$, defined by \eqref{eq: kappa}, and $\lambda$, defined by \eqref{eq: lambda}. The next lemma achieves the desirable goal of linking the two cubics. 

\begin{lem}\label{lem: 67}
If $D\neq D_\bullet$ then $N_0(\kappa)=N_0(\lambda)$.
\end{lem}
\begin{proof}
We chain several transformations that turn $\kappa$ into $\lambda$, while keeping the number of distinct roots unchanged. We start with the formula \eqref{eq: kappa2}
\[
\kappa(x)=(x+2)\left(x-2+\frac{C}{2}\right)^2-(D-D_\bullet)(x-2).
\]
Consider the shifted cubic $\kappa_1(x)=\kappa(x-2)$. Then
\[
\kappa_1(x)=x\left(x+\frac{C-8}{2}\right)^2-(D-D_\bullet)(x-4).
\]
Notice that $x=0$ is not a root of $\kappa_1(x)$. We may thus pass from $\kappa_1$ to its reciprocal $\kappa_2$, given by $\kappa_2(x)=x^3\kappa_1(1/x)$, without changing the root count. We have
\[
\kappa_2(x)=\left(1+\frac{C-8}{2}\cdot x\right)^2-(D-D_\bullet)(x^2-4x^3).
\]
Finally, by using the rescaling
\[
\big(4(D-D_\bullet)\big)^2\kappa_2\left(\frac{x}{4(D-D_\bullet)}\right)=\Big(4(D-D_\bullet)+\frac{C-8}{2}\cdot x\Big)^2-(D-D_\bullet)x^2+x^3
\]
we arrive at $\lambda$. \end{proof}

Moving forward, $\lambda$ could be thought of as the \emph{main cubic}.

\section{An interlude on the $K$ family}\label{sec: K}
Let us take a closer look at the $K$ family, parameterized by 
\begin{align}\label{eq: CDK}
C=K^2+2K, \qquad D=-(2K^3+3K^2).
\end{align}
In particular, we highlight four distinguished $K$ cubics: for $K=0$ we get the Markoff cubic $\M(0,0)$; for $K=-2$ we get the Cayley cubic $\M(0,4)$; for $K=2$ we get the cubic
\begin{align*}
\M(8,-28): \qquad x^2+y^2+z^2=xyz-8(x+y+z)-28;
\end{align*}
and for $K=-1$ we get the cubic
\begin{align*}
\M(-1,-1): \qquad x^2+y^2+z^2=xyz+(x+y+z)-1.
\end{align*}
The latter parameter value, $K=-1$, represents the cusp point of the parameterization \eqref{eq: CDK}. 

The first three $K$ cubics--two familiar, one somewhat unexpected--arise by intersecting the $K$ family with two other outstanding one-parameter families. The $K$ family meets the Fricke family, described by $C=0$, for the parameters values $K=0$ and $K=-2$; these represent the Markoff cubic and the Cayley cubic. Next, let us see how the $K$ family meets the Cayley family, described by $D=D_\bullet$. Note that the level parameter corresponding to $C=K^2+2K$ is $D_\bullet=4-2C-C^2/4=4-4K-3K^2-K^3-K^4/4$; we then compute
\begin{align}\label{KCay}
4(D-D_\bullet)=(K+2)(K-2)^3.
\end{align}
In particular, the $K$ family meets the Cayley family for the parameters values $K=-2$ and $K=2$. The former value represents the Cayley cubic, once again; the latter value brings forth the $\M(8,-28)$ cubic.

\begin{figure}[ht]
    \centering
    \includegraphics[width=0.95\textwidth]{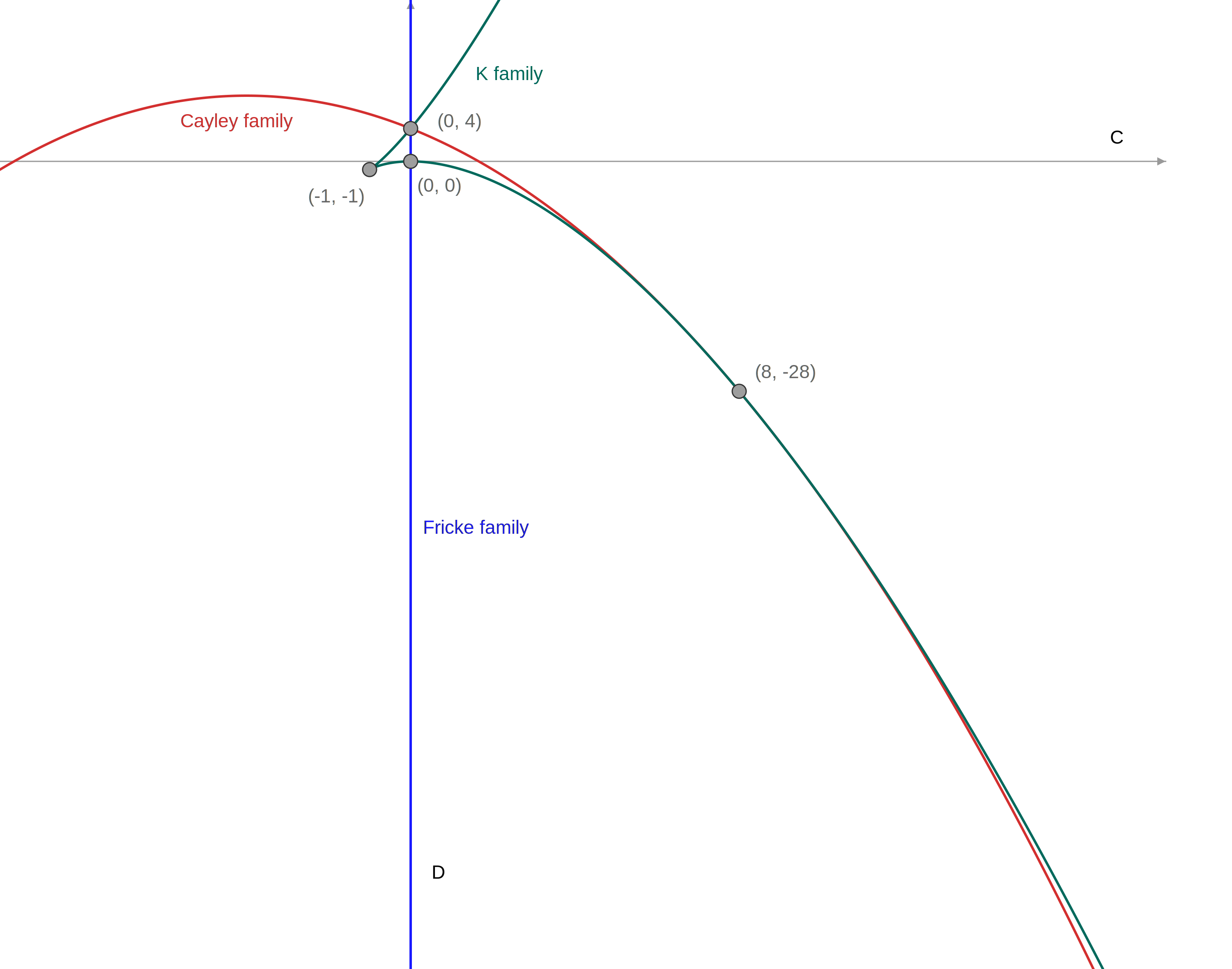}
    \caption{The $K$ family meets the Fricke family ($C=0$) and the Cayley family ($D=4-2C-C^2/4$) in the real $(C,D)$ plane. Note: the axes are not equally scaled.}
    \label{fig: CK}
\end{figure}

It was convenient to introduce the $K$ family in relation to the modified Markov equation \eqref{eq: GSLM}, which was recently studied in the literature. But that is not the main reason why the $K$ family deserves attention for the discussion at hand. Two structural facts make the $K$ family stand out in our study of the Vieta graph for the generalized Markoff equation $\M(C,D)$. By either count, the $K$ family is in fact inevitable. On the one hand, it turns out that the $K$ family naturally comes up in the $N_3$ count; see the next section. On the other hand, and perhaps more subtly, one can arrive at the $K$ family by considering the non-separability of the cubics $\kappa$ and $\lambda$. That $\kappa$ and $\lambda$ have repeated roots in the case of the $K$ family will be apparent in the proof of Lemma~\ref{lem: Kfam}. Conversely, it can be checked that $\kappa$ and $\lambda$ are non-separable if and only if $C$ and $D$ are parameterized as in \eqref{eq: CDK}. The argument is straightforward but somewhat tedious, so we omit it.

\section{The $N_3$ count}\label{sec: N3}
In algebraic geometry, a \emph{node} on a cubic hypersurface $f(x,y,z)=0$ is a point where $\partial_x f=\partial_y f=\partial_z f=0$; a cubic is \emph{smooth} if it has no nodes. Therefore $N_3$ is the nodal count for the generalized Markoff equation $\M(C,D)$. In the Vieta graph of $\M(C,D)$, the $N_3$ count is the number of isolated vertices.

\begin{thm}\label{thm: nodes}
A generalized Markoff cubic has $N_3=0$ nodes--so it is smooth--unless it belongs to the Cayley family or to the $K$ family. 

A cubic in the $K$ family has $N_3=1$ nodes, namely $(-K,-K,-K)$, except for the Cayley cubic $\M(0,4)$. 

A cubic in the Cayley family has $N_3=3$ nodes, namely $(-2, -2, 2-C/2)$, $(-2, 2-C/2, -2)$, and $(2-C/2, -2, -2)$, except for the Cayley cubic and the $\M(8,-28)$ cubic (which has $N_3=1$ nodes). 

The Cayley cubic has $N_3=4$ nodes: $(-2, -2, 2)$, $(-2, 2, -2)$, $(2, -2, -2)$, and $(2, 2, 2)$. 
\end{thm}
\begin{figure}[ht]
    \hspace*{-1.0cm}    
    \includegraphics[width=1.2\textwidth]{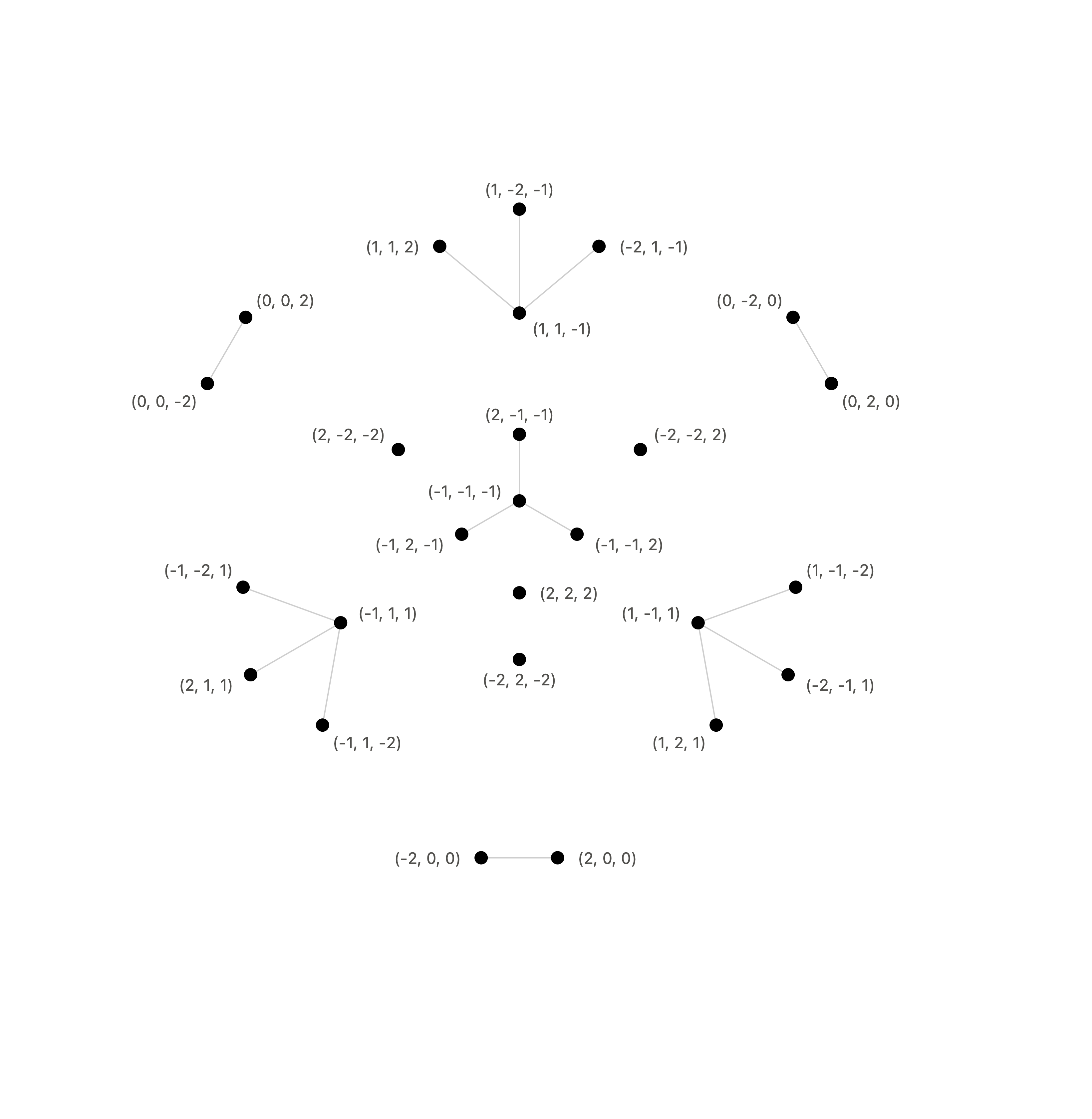}
    \vspace{-3cm}\caption{The Vieta graph of the Cayley cubic $\M(0,4)$ over $\F_{\!5}$. There are $N_3=4$ isolated vertices.}
    \label{fig: C0_D4}
\end{figure}

\begin{proof}
We show that a node can only occur in the Cayley family and in the $K$ family; and that when it does occur, a node has one of the following two types: 
\begin{itemize}
\item[(i)] $(-2, -2, 2-C/2)$, $(-2, 2-C/2, -2)$, $(2-C/2, -2, -2)$ in the Cayley family;
\item[(ii)] $(-K, -K, -K)$ in the $K$ family.
\end{itemize} 
Thus, generically, a cubic in the Cayley family has three nodes, and a cubic in the $K$ family has one node. The exceptions are the common cubics. For the Cayley cubic $\M(0,4)$, the four nodes of type (i) and (ii) coexist. For the $\M(8,-28)$ cubic, they degenerate into a single node $(-2, -2, -2)$.

Consider the polynomial $f$ of \eqref{eq: mainf}. The nodes are those solutions $(x,y,z)$ in the $N_2$ count that satisfy, in addition, $(\partial_x f)(x,y,z)=0$; that is to say, they satisfy 
\[2x=yz-C. \tag{$*$}\] 
We thus have to impose $(*)$ on the various solutions that we have found in the proof of Theorem~\ref{thm: N2}.

Assume $D=D_\bullet$. Among the solutions of the form $(-2,y, -y-C/2)$, those that satisfy $(*)$ occur for $y=-2$ or $y=2-C/2$. We thus have the nodes $(-2, -2, 2-C/2)$ and $(-2, 2-C/2, -2)$. If $C\neq 0,8$, then the solution $(2-C/2, -2, -2)$ satisfies $(*)$, so it is another node. If $C=0$, then among the solutions of the form $(2,y,y)$, only $(2,2,2)$ and $(2,-2,-2)$ satisfy $(*)$. In summary, for all values of $C$ we get the nodes $(-2, -2, 2-C/2)$, $(-2, 2-C/2, -2)$, $(2-C/2, -2, -2)$. This checks out part (i). The value $C=0$ also has $(2,2,2)$ as an additional node; this agrees with part (ii) for $K=-2$. 

Assume now that $D\neq D_\bullet$. If $C=0$, then the solutions coming from the $N_2$ count take the form $(x_0,0,0)$ where $x_0^2=D$. Imposing $(*)$ determines  $x_0=0$, so we only get the node $(0,0,0)$ when $D=0$. This case agrees with part (ii) for $K=0$. 

If $C\neq 0$, then we are looking for a node among solutions of the form $(x_0,C/(x_0-2), C/(x_0-2))$ with $x_0\neq \pm 2$. Imposing $(*)$ amounts to 
\[2x_0=\frac{C^2}{(x_0-2)^2}-C,\]
equivalently $(x_0-2+C/2)(x_0^2-2x_0-C)=0$. Note, however, that $x_0$ is a root of the cubic $\kappa(x)$. The alternate expression \eqref{eq: kappa2} clearly rules out the possibility that $x_0=2-C/2$; it also ensures that $x_0\neq \pm 2$. Therefore $x_0^2-2x_0-C=0$. This relation implies, incidentally, that $C/(x_0-2)=x_0$, so the desired triple is fully symmetric, $(x_0,x_0,x_0)$.

Dividing $\kappa(x)$ by $x^2-2x-C$ yields $(C-D)x+2(C^2+D)$ as a remainder. Consequently, we obtain the system
\[\begin{cases}
x_0^2-2x_0-C=0,\\
(C-D)x_0+2(C^2+D)=0.
\end{cases}\]

The case $D=C$ forces $C=-1$ and then $x_0=1$. Thus, for the parameters $C=D=-1$ we get $(1,1,1)$ as a node. This case agrees with part (ii) for $K=-1$. Note that $D\neq D_\bullet$ does hold for $C=D=-1$.

The case $D\neq C$ gives 
\[x_0=\frac{2(D+C^2)}{D-C}=2+\frac{2C(C+1)}{D-C}.\]
Now $x_0^2-2x_0-C=0$ can be brought to a quadratic equation in $D$:
\[D^2-2(3C+2)D-(4C^3+3C^2)=0.\]
The discriminant is $16(C+1)^3$. In order for the quadratic to have solutions, $C+1$ must be a square, say $C+1=E^2$ for some $E\in \Fp$. The solutions are $D=3E^2-1 \pm 2E^3$. We may choose $D=3E^2-1- 2E^3$; the other choice of sign can be recovered via $E:=-E$. Next, we find $x_0=1-E$. So, to summarize, we get the node $(1-E, 1-E, 1-E)$ when $C=E^2-1$ and $D=3E^2-1- 2E^3$ for some $E\in \Fp$. The substitution $E:=K+1$ gives $C=K^2+2K$ and $D=-(2K^3+3K^2)$, the parametrization of the $K$ family; as for the node, it takes the form $(-K,-K,-K)$. This completes the verification of part (ii).

Strictly speaking, a few values of the parameter $E\in \Fp$ have to be ruled out here. The requirements $C\neq 0$,  $D\neq C$, $D\neq D_\bullet$ amount to $E\neq -1, 0, 1, 3$. In terms of $K$, we need $K\neq -2, -1, 0, 2$. We have already seen that the values $K=-2, -1, 0$ satisfy part (ii). The case $K=2$ is in agreement as well: the corresponding cubic $\M(8,-28)$ belongs to the Cayley family, in which case part (i) degenerates to one symmetric node $(-2, -2, -2)$. \end{proof}

It is well known that the Cayley cubic is distinguished among cubic surfaces as having the most nodes, namely $4$. Cubics along the Cayley family share the feature of having a relatively large number of nodes. The exception is the $\M(8,-28)$ cubic, which seems to counterbalance the nodal excess of the Cayley cubic. 

\begin{figure}[ht]
    \centering
    \includegraphics[width=0.9\textwidth]{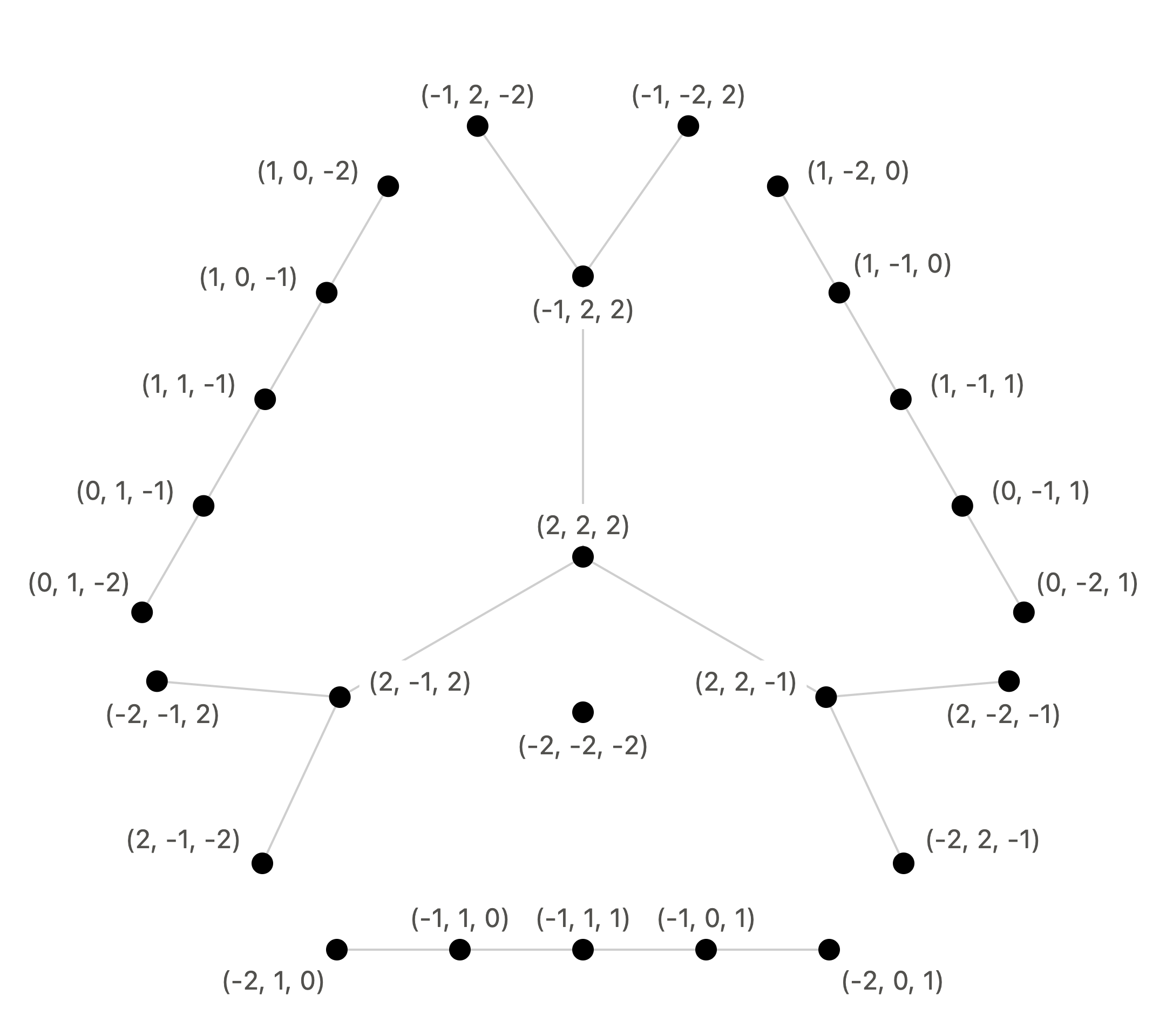}
    \caption{The Vieta graph of the cubic $\M(8,-28)$ over $\F_{\!5}$. The only isolated vertex is $(-2,-2,-2)$.}
    \label{fig: C3_D2}
\end{figure}

\section{Degree distributions: the Cayley family and the $K$ family}
We can finally derive the explicit degree distribution for some distinguished families of $3$-dimensional Vieta graphs. Let us recall from Lemma~\ref{lem: Nk} that the total number of vertices, the number of deficient vertices, and the degree distribution in a $3$-dimensional Vieta graph can be obtained from the $N_k$ counts--that is, $N_0$, $N_1$, $N_2$, and $N_3$--as in the table below. The formulas for $N_0$, $N_1$, $N_2$, and $N_3$ are given in Theorem~\ref{thm: v}, Theorem~\ref{thm: N1}, Theorem~\ref{thm: N2} and Theorem~\ref{thm: nodes} respectively.

\begin{table}[ht]\renewcommand{\arraystretch}{1.2}
 \begin{tabular}{| c || c |} 
 \hline
total & $N_0$  \\
deficient & $3N_1-3N_2+N_3$  \\
 \hline
degree $2$ & \: $3(N_1-2N_2+N_3)$ \:  \\
degree $1$ & $3(N_2-N_3)$ \\
degree $0$ & $N_3$  \\
 \hline
\end{tabular}  \bigskip\caption{Degree counts from $N_k$ counts in dimension $3$.}\end{table}

In this section, we start off with the easiest one-parameter families: the Cayley family and the $K$ family. The Fricke family turns out to be trickier, and will be dealt with later.  

\subsection{The Cayley family} Recall, this the family described by $D=D_\bullet$; the corresponding generalized Markoff equation is \eqref{eq: C}. Two parameter values play a special role: $C=0$ and $C=8$. They describe the Cayley cubic $\M(0,4)$, respectively the $\M(8,-28)$ cubic.

For this family, we already have all the $N_k$ counts. We collect them below.

\begin{countlem}
For the Cayley family of cubics, the $N_k$ counts are given as follows:
\begin{align*}
N_0&=p^2+1+\sigma(C^2-8C)\cdot p,\\
N_1&=\big(3+\llbracket C=0\rrbracket \big) p-4+\llbracket C=8\rrbracket,\\
N_2&= \big(1+\llbracket C=0\rrbracket\big)p+1-\llbracket C=0\rrbracket-\llbracket C=8\rrbracket,\\
N_3&=3+\llbracket C=0\rrbracket-2\llbracket C=8\rrbracket.
\end{align*}
\end{countlem}

We obtain the following result.

\begin{thm}\label{thm: CayleyDetailed} Consider the Vieta graph of a cubic in the Cayley family. The total number of vertices, the number of deficient vertices, and the degree distribution for the deficient vertices is given in Table~\ref{table Cayley} below.

\begin{table}[ht]\renewcommand{\arraystretch}{1.2}
 \begin{tabular}{| c || c | c | c |} 
  \hline
  &  $C=0$ &  $C=8$ &  $C\neq 0,8$  \\
 \hline
total & $p^2+1$   & \: $p^2+1$  \: & \: $p^2+1\pm p$ \: \\
deficient & $6p-8$ & \: $6p-8$  \: & \: $6p-12$ \:  \\
 \hline
degree $2$ & \: $0$  \: & \: $3(p-2)$ \: & \:  $3(p-3)$ \: \\
degree $1$ & \: $6(p-2)$  \: & \: $3(p-1)$ \: & \:  $3(p-2)$ \:\\
degree $0$ & \: $4$  \: & \: $1$ \: & \:  $3$ \:\\
 \hline
 \end{tabular} \bigskip\caption{Degree distribution in Vieta graphs for cubics in the Cayley family.}\label{table Cayley}
\end{table}
\end{thm}

Two interesting outcomes are noticeable in Table~\ref{table Cayley}. Firstly, the Cayley cubic $\M(0,4)$ and the $\M(8,-28)$ cubic are alike in that they have the same number of vertices, and the same number of deficient vertices. Secondly, the Cayley cubic has no vertices of degree $2$. There is a handful of Vieta graphs for the equation $\M(C,D)$ with this property--for example $\M(0,0)$ and $\M(0,2)$ over $\F_{\!5}$, and $\M(3,6)$ and $\M(0,5)$ over $\F_{\!7}$--but these can only occur for $p\leq 17$. The fact here is that the Cayley cubic is the only generalized Markoff cubic whose Vieta graph has no vertices of degree $2$, independently of $p$. 

\begin{rem}
The number of deficient vertices in the Cayley family is unusually high. The Vieta graph of a generalized Markoff cubic has $p^2+O(p)$ vertices; it has $6p+O(1)$ deficient vertices if it belongs to the Cayley family but otherwise it has only about half as many deficient vertices--namely, $3p+O(\sqrt{p})$. 
\end{rem}

\subsection{The $K$ family} 
Recall, this family is the one-parameter family given by $C=K^2+2K$ and $D=-(2K^3+3K^2)$; the corresponding generalized Markoff equation is \eqref{eq: KMar}. For $K=-2$ and $K=2$ we get cubics that also belong to the Cayley family, discussed above, namely the Cayley cubic ($C=0$) and the $\M(8,-28)$ cubic ($C=8$). In what follows, we put aside these parameter values. Two other special parameter values are $K=0$ and $K=-1$.

\begin{countlem}\label{lem: Kfam} Consider the $K$ family of cubics, where $K\neq \pm 2$. Then the $N_k$ counts are given as follows:
\begin{align*}
N_0&= p^2+1+\big(3-\llbracket K=-1\rrbracket\big)\sigma(K^2-4)\cdot p, \\
N_1&= p-3-\sigma(-K-1),\\
N_2&=2-\llbracket K=-1\rrbracket-\llbracket K=0\rrbracket,\\
N_3&=1.
\end{align*}
\end{countlem}

\begin{proof} Recall, $K\neq \pm 2$ means that we are in the non-Cayley case $D\neq D_\bullet$. In fact $4(D-D_\bullet)=(K+2)(K-2)^3$, as pointed out in \eqref{KCay}.

The formula for $N_0$ comes from Theorem~\ref{thm: v}(ii). Here $\sigma(D-D_\bullet)=\sigma(K^2-4)$, and $\kappa(x)=(x+K^2-2)(x+K)^2$. Thus $\kappa$ has $-K$ and $2-K^2$ as roots, implying that $N_0(\kappa)=2-\llbracket K=-1\rrbracket$.

The formula for $N_1$ comes from Theorem~\ref{thm: N1}(ii). Just like $\kappa$ above, $\lambda$ also has a simple split formula: $\lambda(x)=\big(x+(K-2)^2\big)\big(x+(K+2)(K-2)^2\big)^2$. We may then evaluate, very easily, that
\begin{align*}
\sum_{x\in \Fp}\sigma(\lambda(x))&=\sum_{x\neq -(K+2)(K-2)^2}\sigma\big(x+(K-2)^2\big)\\
&=-\sigma\big(-(K+2)(K-2)^2+(K-2)^2\big)=-\sigma(-K-1).
\end{align*}

The formula for $N_2$ comes from Theorem~\ref{thm: N2}(ii). Note that $C=0$ if and only if $K=0$. As we have already seen, $N_0(\kappa)=2-\llbracket K=-1\rrbracket$. 

That $N_3=1$ is part of Theorem~\ref{thm: nodes}.
\end{proof}

The above proof highlights that $\kappa$ and $\lambda$ are non-separable for the $K$ family of cubics. Furthermore, $\kappa$ and $\lambda$ have a triple root if and only if $K=-1$. (Note: $\kappa$ and $\lambda$ are only considered in the non-Cayley case $K\neq \pm 2$.) 

\begin{thm}
Consider the Vieta graph of a cubic in the $K$ family, where $K\neq \pm 2$. The total number of vertices, the number of deficient vertices, and the degree distribution for the deficient vertices is given in Table \ref{T3} below.

\begin{table}[ht]\renewcommand{\arraystretch}{1.2}
 \begin{tabular}{| c || c | c | c |} 
  \hline
 &  $K=0$ &  $K=-1$ & $K\neq 0,-1, \pm 2$ \\
 \hline 
\rule{0pt}{2.5ex}total & \:$p^2+1+3\sigma(-1)p$\:  & \:$p^2+1+2\sigma(-3)p$\: & \:$p^2+1+3\sigma(K^2-4)p$\: \\
\rule{0pt}{2.5ex}deficient &  $3p-11-3\sigma(-1)$ & $3p-11$ & $3p-14-3\sigma(-K-1)$ \rule[-1ex]{0pt}{0pt}  \\ \hline
\rule{0pt}{2.5ex}degree $2$ & $3\big(p-4-\sigma(-1)\big)$  & $3(p-4)$ &  $3\big(p-6-\sigma(-K-1)\big)$  \\
degree $1$ & $0$  & $0$ & $3$ \\
degree $0$ &  $1$  & $1$ &  $1$ \\
 \hline
 \end{tabular} \bigskip\caption{Degree distribution in Vieta graphs for cubics in the $K$ family.}\label{T3}
\end{table}
\end{thm}

In Corollary~\ref{cor: no01}, we characterized the absence of isolated and leaf vertices in the Vieta graph of $\M(C,D)$. As displayed in Table~\ref{T3}, a noticeable feature of the Vieta graph for the Markoff cubic $\M(0,0)$ (the case $K=0$) and for the $\M(-1, -1)$ cubic (the case $K=-1$), is that they have no leaves yet they do have isolated vertices. It is not hard to check that they are the only Vieta graphs for generalized Markoff cubics having this feature.

\begin{figure}[ht]
    \centering
    \includegraphics[width=0.9\textwidth]{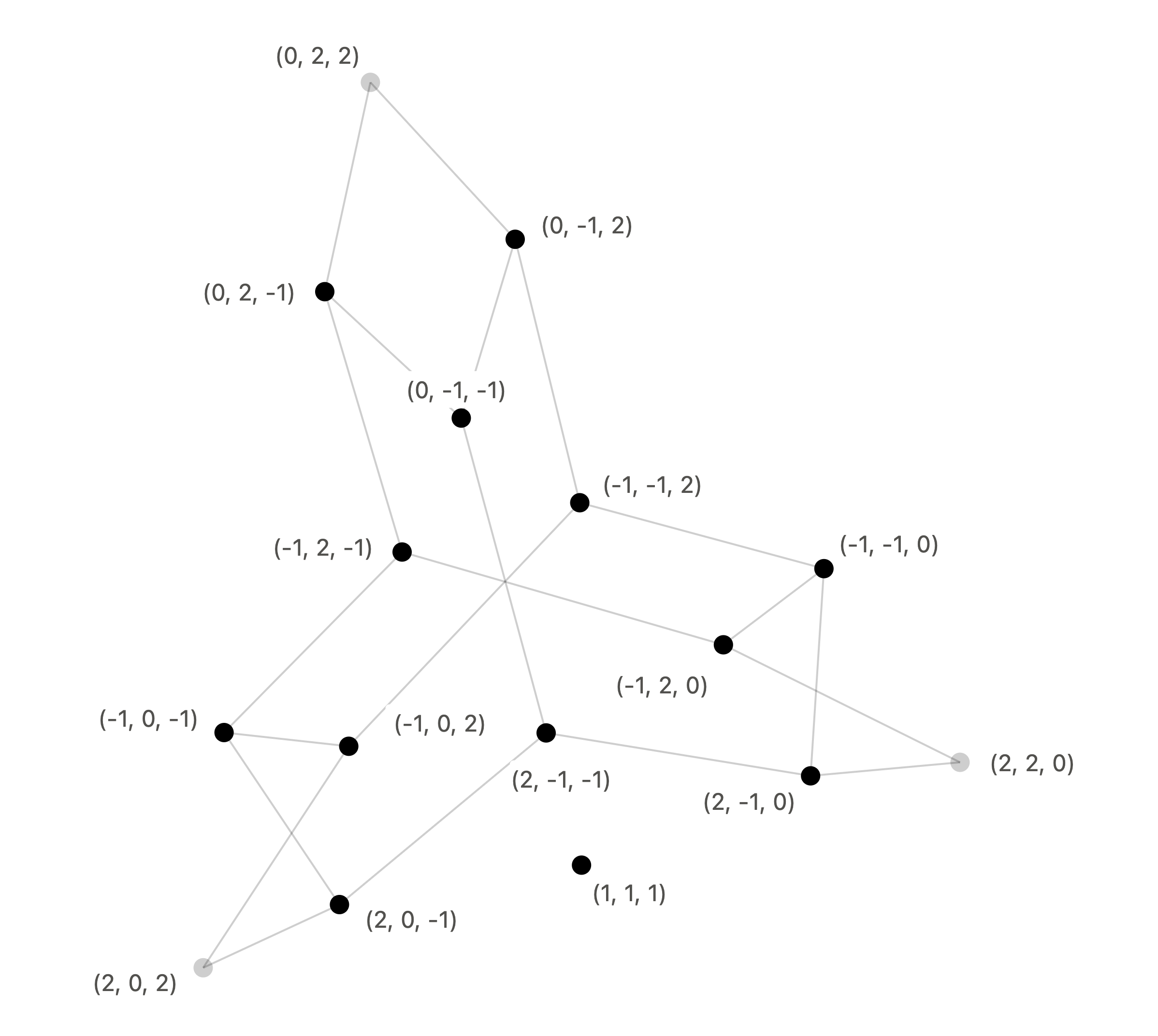}
    \caption{The Vieta graph of the cubic $\M(-1,-1)$ over $\F_{\!5}$. The deficient vertices are all of degree $2$ (highlighted in light grey), except for the one isolated vertex.}
    \label{fig: Cm1_Dm1}
\end{figure}

\section{An interlude on Jacobsthal sums}\label{sec: Jac}
In the non-Cayley case $D\neq D_\bullet$, the main difficulty in obtaining an explicit degree distribution for the Vieta graph of $\M(C,D)$ is that of computing the quadratic character sum 
\begin{align}\label{eq: charsum}
\sum_{x\in \Fp}\sigma(\lambda(x))
\end{align}
where $\lambda(x)$ is the cubic defined in \eqref{eq: lambda}.

The problem of computing a quadratic character sum with a cubic argument, of the form \eqref{eq: charsum}, is very important and very difficult in general. If $\lambda(x)$ is non-separable, then the evaluation of the sum \eqref{eq: charsum} is straightforward. Let us recall that non-separability of $\lambda(x)$ is one of the motivations for considering the $K$ family of generalized Markoff cubics. 

In the separable case, fairly simple explicit evaluations of the sum \eqref{eq: charsum}  are known when $\lambda(x)$ is (up to a variable shift) of the form $x^3+bx$ or $x^3+c$. The following is a brief overview of the cubic Jacobsthal sums
\begin{align}
\varphi_2(b)&=\sum_{x\in \Fp} \sigma(x^3+bx) \qquad (b\in \Fp^*),\\
\psi_3(c)&=\sum_{x\in \Fp} \sigma(x^3+c)\qquad (c\in \Fp^*). 
\end{align}
The indexing owes to the fact that these quadratic character sums with cubic arguments have higher degree analogues, see \cite[Chapter 3]{N}.

The Jacobsthal sums $\varphi_2(b)$ and $\psi_3(c)$ can be evaluated in terms of decompositions of $p$ into sums of squares. The main evaluations are those for $\varphi_2(1)$ and $\psi_3(1)$. The following two lemmas collect relevant results from \cite[Chapter 3]{N}.

\begin{lem}\label{lem: J1} The following hold:
\begin{itemize}
\item[(i)] $\varphi_2(s^2b)=\sigma(s)\varphi_2(b)$ for $s\in \Fp^*$;
\item[(ii)] if $p\equiv 3$ mod $4$ then $\varphi_2(b)=0$;
\item[(iii)] if $p\equiv 1$ mod $4$, let $A=A_2(p)$ be the integer defined by writing $p=A^2+B^2$ with $A\equiv -1$ mod $4$; then $\varphi_2(1)=2A_2(p)$.
\item[(iv)] $\varphi_2(-1)=\sigma(2)\varphi_2(1)$.
\end{itemize}
\end{lem}

\begin{lem}\label{lem: J2} The following hold:
\begin{itemize}
\item[(i)] $\psi_3(s^3c)=\sigma(s)\psi_3(c)$ for $s\in \Fp^*$;
\item[(ii)] if $p\equiv 2$ mod $3$ then $\psi_3(c)=0$;
\item[(iii)] if $p\equiv 1$ mod $3$, let $A=A_3(p)$ be the integer defined by writing $p=A^2+3B^2$ with $A\equiv -1$ mod $3$; then $\psi_3(1)=2A_3(p)$.
\end{itemize}
\end{lem}

\section{Degree distributions: the Fricke family} 
Recall, the Fricke family is defined by $C=0$; the corresponding generalized Markoff equation is \eqref{eq: G0}. We discard the parameter value $D=4$, which yields the Cayley cubic. This is the only intersection of the Fricke family with the Cayley family.

Before stating the next counting lemma, it is useful to introduce the following quadratic character sum, depending on a parameter $a\in \Fp$:
\begin{align}\label{eq: Sa}
S(a)=\sum_{x\in \Fp}\sigma(x^3+ax^2+x).
\end{align}

\begin{countlem}\label{lem: Frickefam} Consider the Fricke family of cubics $\M(0,D)$, where $D\neq 4$. Then the $N_k$ counts are given as follows:
\begin{align*}
N_0&= p^2+1+\big(3+\sigma(D)\big)\sigma(D-4)\cdot p, \\
N_1&= p-3+\sigma(D-4)\cdot S\bigg(\frac{2(D+4)}{D-4}\bigg),\\
N_2&=1+\sigma(D),\\
N_3&=\llbracket D=0\rrbracket.
\end{align*}
\end{countlem}

\begin{proof} We are in the non-Cayley case. The Cayley level for $C=0$ is $D_\bullet=4$. 

By Theorem~\ref{thm: v}(ii), 
\[N_0=p^2+1+\big(1+N_0(\kappa)\big)\sigma(D-4)\cdot p\] 
where $\kappa(x)=x^3-2x^2-Dx+2D=(x-2)(x^2-D)$. Note that $N_0(\kappa)=2+\sigma(D)$.

The formula for $N_2$ draws from Theorem~\ref{thm: N2}(ii), up to the following rewriting. We have $\lambda(x)=x^3-(D-4)x^2+16(x-(D-4))^2$. The shift $x:=x+(D-4)$ gives
\[\sum_{x\in \Fp}\sigma(\lambda(x))=\sum_{x\in \Fp}\sigma\Big(x^3+2(D+4)x^2+(D-4)^2x\Big).\]
In the latter sum we rescale $x:=(D-4)x$. We get the desired rewriting
\[\sum_{x\in \Fp}\sigma(\lambda(x))=\sigma(D-4)\cdot S\bigg(\frac{2(D+4)}{D-4}\bigg).\]

By Theorem~\ref{thm: N2}(ii), $N_2=N_0(\kappa)-1=1+\sigma(D)$.

By Theorem~\ref{thm: nodes}, $N_3=0$ unless the Fricke cubic happens to belong to the $K$ family as well, in which case $N_3=1$. The only cubic common to both families, other than the already discarded Cayley cubic, is the Markoff cubic occurring at $D=0$. Thus $N_3=\llbracket D=0\rrbracket$.
\end{proof}

The number of vertices in the Vieta graph of a Fricke cubic is explicitly given by $N_0$. In order to obtain the number of deficient vertices, or the degree distribution, we need to evaluate a quadratic character sum of the form $S(a)$. The following lemma offers some partial answers in terms of the Jacobsthal sum $\varphi_2(1)$--which, we recall, has an explicit evaluation.

\begin{lem}\label{lem: Sa}
Let $S(a)$ be defined as in \eqref{eq: Sa}. Then the following hold:
\begin{itemize}
\item[(i)] $S(-a)=\sigma(-1)S(a)$; 
\item[(ii)] $S(0)=\varphi_2(1)$;
\item[(iii)] $S(2)=-\sigma(-1)$;
\item[(iv)] $S(6)=\sigma(2)\varphi_2(1)$.
\end{itemize}
\end{lem}

The upshot of this lemma is that $S(a)$ is explicitly computable for $a=0, \pm 2, \pm 6$.

\begin{proof}
(i) Immediate by using the sign change $x:=-x$. (ii) True by notation. (iii) This is easy:
\begin{align*}
S(2)=\sum_{x\in \Fp} \sigma(x^3+2x^2+x)=\sum_{x\in \Fp} \sigma\big(x(x+1)^2\big)=\sum_{x\neq -1} \sigma(x)=-\sigma(-1).
\end{align*}
 (iv) This identity is more involved. For the reader's benefit we give a fairly self-contained argument, and we refer the reader to \cite[pp.93-96, p.117]{N} for more context. We have
 \begin{align*}
 S(6)=\sum_{x\in \Fp^*} \sigma\big(x(x^2+6x+1)\big)=\sum_{x\in \Fp^*} \sigma\big(x^{-1}(x^2+6x+1)\big)=\sum_{y\in \Fp} \sigma(y)\cdot N(y)
  \end{align*}
  where $N(y)$ is the number of solutions $x\in \Fp^*$ to $y=x^{-1}(x^2+6x+1)$. This equation amounts to the quadratic $x^2+(6-y)x+1=0$, which has $1+\sigma\big((6-y)^2-4\big)$ (non-zero) solutions in $\Fp$. Thus $N(y)=1+\sigma(y^2-12y+32)$. We continue:
\begin{align*}
S(6)=\sum_{y\in \Fp} \sigma(y) \Big(1+\sigma(y^2-12y+32)\Big)=\sum_{y\in \Fp} \sigma\big(y(y^2-12y+32)\big).
\end{align*}
The rescaling $y:=4y$, followed by the shift $y:=y+1$ gives
\begin{align*}
S(6)=\sum_{y\in \Fp} \sigma(y^3-3y^2+2y)=\sum_{y\in \Fp} \sigma(y^3-y)=\varphi_2(-1).
\end{align*}
Finally, we recall from Lemma~\ref{lem: J1} (iv) that $\varphi_2(-1)=\sigma(2)\varphi_2(1)$.
\end{proof}

Imposing $\dfrac{2(D+4)}{D-4}=0, \pm 2, \pm 6$ amounts to four parameter values: $D=-4$, $D=0$, $D=2$, and $D=8$. The $N_k$ counts can therefore be worked out explicitly for these values. We remark that the computations of $N_1$ is also discussed in \cite[Example 5.15]{N}.

We deduce the following.

\begin{thm}
Consider the Vieta graph of a Fricke cubic $\M(0,D)$, where $D=-4, 0, 2, 8$. The total number of vertices, the number of deficient vertices, and the degree distribution for the deficient vertices are given in Tables \ref{T4} and \ref{T5} below.

\begin{table}[ht]\renewcommand{\arraystretch}{1.2}
 \begin{tabular}{| c || c | c |} 
  \hline
 &  $D=0$ &  $D=-4$ \\
 \hline
\rule{0pt}{3ex}total & \:$p^2+1+3\sigma(-1)p$\:  & $p^2+1+\big(3+\sigma(-1)\big)\sigma(-2)p$ \\
\rule{0pt}{2.5ex}deficient &  $3p-11-3\sigma(-1)$ & $3\big(p-4-\sigma(-1)+\sigma(-2)\varphi_2(1)\big)$ \rule[-1.5ex]{0pt}{0pt}  \\ \hline
\rule{0pt}{2.5ex}degree $2$ & $3\big(p-4-\sigma(-1)\big)$  & \:$3\big(p-5-2\sigma(-1)+\sigma(-2)\varphi_2(1)\big)$\:  \\
\rule{0pt}{1.5ex}degree $1$ & $0$  & $3\big(1+\sigma(-1)\big)$ \\
\rule{0pt}{1.5ex}degree $0$ &  $1$  & $0$ \\
 \hline
 \end{tabular} \bigskip\caption{Degree distribution in Vieta graphs for the Fricke cubics $\M(0,0)$ and $\M(0,-4)$.}\label{T4}
\end{table}

\begin{table}[ht]\renewcommand{\arraystretch}{1.2}
 \begin{tabular}{| c || c | c |} 
  \hline
 &  $D=2$ &  $D=8$ \\
 \hline
\rule{0pt}{2.5ex}total & \:$p^2+1+\big(3+\sigma(2)\big)\sigma(-2)p$\:  & $p^2+1+\big(3+\sigma(2)\big)p$ \\
\rule{0pt}{2.5ex}deficient &  $3\big(p-4-\sigma(2)+\varphi_2(1)\big)$ & $3\big(p-4-\sigma(2)+\sigma(2)\varphi_2(1)\big)$ \rule[-1.5ex]{0pt}{0pt}  \\ \hline
\rule{0pt}{2.5ex}degree $2$ & \:$3\big(p-5-2\sigma(2)+\varphi_2(1)\big)$\:  & $3\big(p-5-2\sigma(2)+\sigma(2)\varphi_2(1)\big)$  \\
\rule{0pt}{1.5ex}degree $1$ & $3\big(1+\sigma(2)\big)$  & $3\big(1+\sigma(2)\big)$ \\
\rule{0pt}{2ex}degree $0$ &  $0$  & $0$ \\
 \hline
 \end{tabular} \bigskip\caption{Degree distribution in Vieta graphs for the Fricke cubics $\M(0,2)$ and $\M(0,8)$.}\label{T5}
\end{table}
\end{thm}

The case $D=0$, the Markoff cubic, is already tabulated in Table \ref{T3} under $K=0$. We include it here in order to see that, by comparison, the degree distribution in the other three cases, $D=-4, 2, 8$, is more complicated. We also note that in the cases $D=-4, 2, 8$, the counts for vertices of degree $0$ or $1$  are in agreement with part (i) of Corollary~\ref{cor: no01}.

The counts in Tables \ref{T4} and \ref{T5} could be made even more explicit by splitting further into two cases: $p\equiv 3$ mod $4$, in which case $\varphi_2(1)=0$ and $\sigma(-1)=-1$, respectively $p\equiv 1$ mod $4$, in which case $\varphi_2(1)=2A_2(p)$ and $\sigma(-1)=1$.

\begin{figure}[ht]
    \hspace*{-0.5cm} 
       \includegraphics[width=1.05\textwidth]{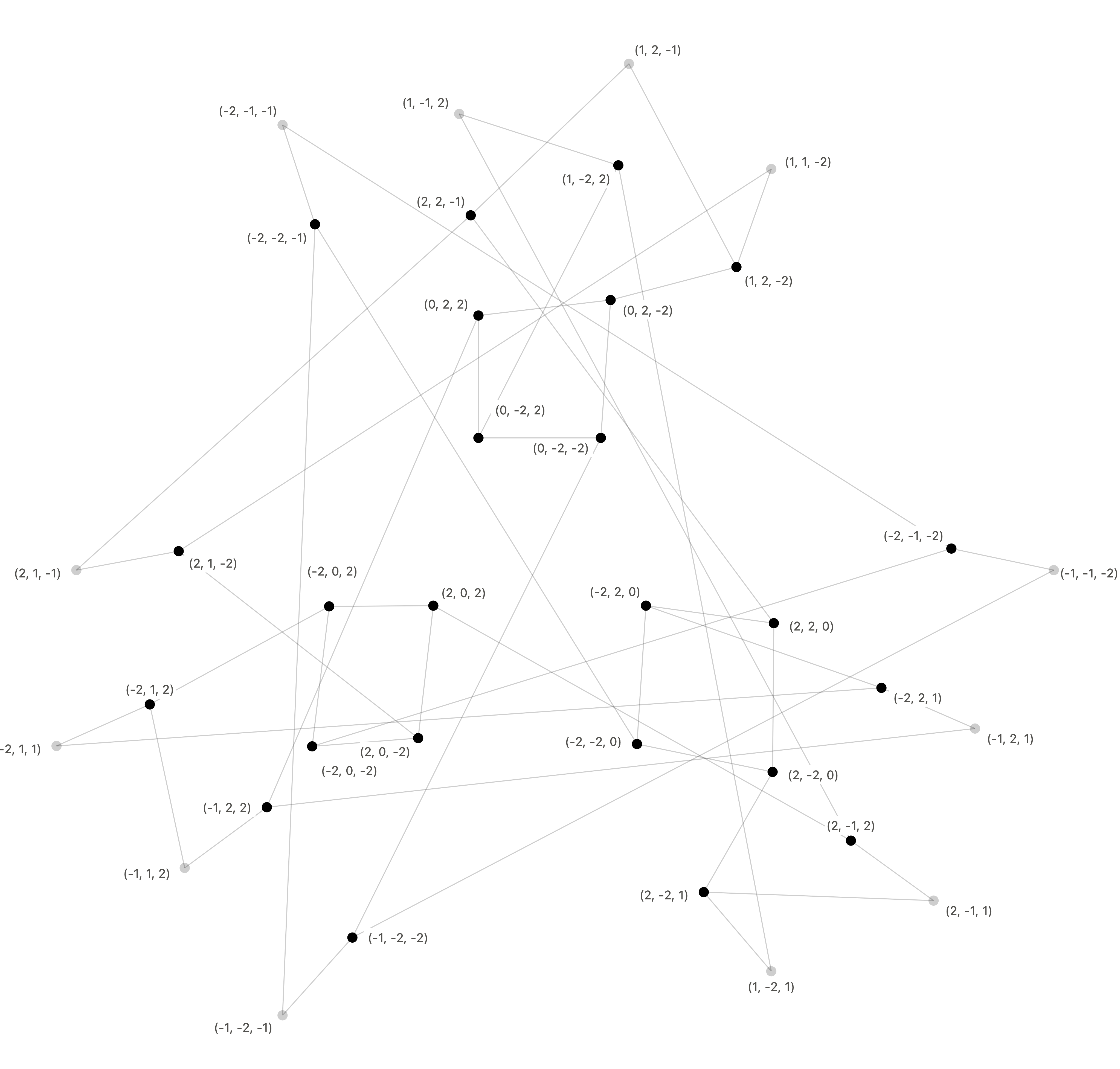}
    \caption{The Vieta graph of the Fricke cubic $\M(0,8)$ over $\F_{\!5}$. This is a connected graph on $36$ vertices; there are $12$ deficient vertices, all of degree $2$ (highlighted in light grey).}
    \label{fig: C0_D8}
\end{figure}

\section{Degree distributions: the $J$ family}
The $J$ family is a one-parameter family of generalized Markoff equations that arise by imposing that the main cubic $\lambda(x)$ is, up to a variable shift, of the form $x^3+c$. Now, a general cubic $x^3+a_2x^2+a_1x+a_0$ is, up to a variable shift, of the form $x^3+c$, if and only if $a_2^2=3a_1$. Straightforward algebra leads to the fact that, in the case of our main cubic $\lambda(x)$, the latter happens if and only if $C$ and $D$ are parameterized by
\begin{align}\label{eq: J}
C=3J^2-4, \quad D=9J^3-36J-28
\end{align}
for $J\in \Fp$. (We remark that one could pursue the related question, as to when $\lambda(x)$ is, up to a variable shift, of the form $x^3+bx$. The algebraic condition on the coefficients of $\lambda(x)$ is, however, more complicated. One can find parametric solutions, but no complete description such as \eqref{eq: J} seems available.) 

For $C$ and $D$ given by \eqref{eq: J}, we compute 
\begin{align}\label{eq: Jdiff}
4(D-D_\bullet)=9(J-2)(J+2)^3.
\end{align}
In particular, the $J$ family meets the Cayley family $D=D_\bullet$ when $J=\pm 2$; the common cubic is $\M(8,-28)$. The $J$ family meets the $K$ family when $J=\pm 2$ (correspondingly, $K=-1$) or $J=-1$ (correspondingly, $K=2$); the common cubics are $\M(8,-28)$ and $\M(-1,-1)$. In the real picture below, the trace of the $J$ family runs strikingly close to the trace of the $K$ family between the intersection points $(-1, -1)$ and $(8,-28)$. We clarify that, despite appearances, the Markoff cubic at $(0,0)$  does not belong to the $J$ family. 

\begin{figure}[ht]
    \centering
    \includegraphics[width=0.90\textwidth]{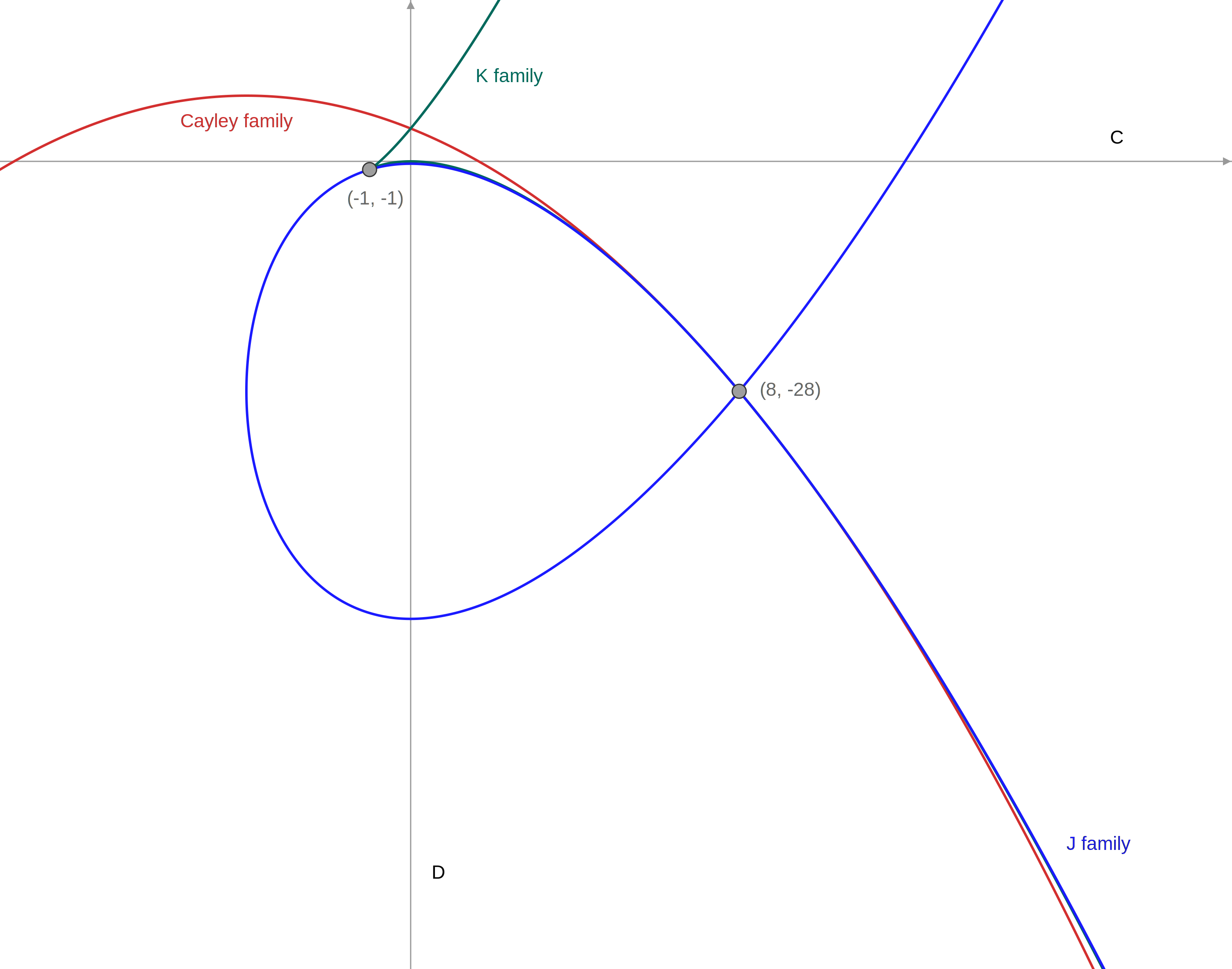}
    \caption{The $J$ family meets the Cayley family and the $K$ family in the real $(C,D)$ plane. Note: the axes are not equally scaled.}
    \label{fig: CKJ}
\end{figure}

\begin{countlem}\label{lem: Jfam} Consider the $J$ family of cubics, where $J\neq \pm 2, -1$. Then the $N_k$ counts are given as follows:
\begin{align*}
N_0&= p^2+1+\big(1+N_0(\lambda)\big)\sigma(J^2-4)\cdot p, \\
N_1&= p-3+\sigma\big(3(J-2)\big)\psi_3\bigg(\frac{J+1}{J-2}\bigg),\\
N_2&=N_0(\lambda)-\llbracket 3J^2-4=0\rrbracket,\\
N_3&=0,
\end{align*}
where 
\begin{align}\label{eq: Jlam}
N_0(\lambda)=\begin{cases}
1 & \textrm{ if } p\equiv 2 \textrm{ mod }3,\\
3\left\llbracket \dfrac{J+1}{J-2}\in (\Fp^*)^3\right\rrbracket & \textrm{ if } p\equiv 1 \textrm{ mod }3.
\end{cases}
\end{align}
\end{countlem}

\begin{proof} We are in the non-Cayley case $D\neq D_\bullet$, since $J\neq \pm 2$. In this proof, we only work with the main cubic $\lambda(x)$; we recall that $N_0(\kappa)=N_0(\lambda)$, by Lemma~\ref{lem: 67}. Using \eqref{eq: Jdiff}, we see that
\[\lambda(x)=x^3-9(J-2)(J+2)^2x^2+27(J-2)^2(J+2)^4x+81(J-2)^2(J+2)^6.\]
Rescaling the variable leads to
\begin{align*}
\frac{\lambda\big((J-2)(J+2)^2x\big)}{\big((J-2)(J+2)^2\big)^3}=x^3-9x^2+27x+\frac{81}{J-2}=(x-3)^3+27\cdot \frac{J+1}{J-2}.
\end{align*}
We note that the latter term is non-zero. 

This formula gives two pieces of information. Firstly, $N_0(\lambda)$ equals the number of roots to the cubic equation $x^3= \dfrac{J+1}{J-2}$; whence formula \eqref{eq: Jlam}. Secondly, by a change of variable $x:=(J-2)(J+2)^2x$, we can evaluate the character sum
\begin{align*}
\sum_{x\in \Fp} \sigma(\lambda(x))&= \sigma\Big(\big((J-2)(J+2)^2\big)^3\Big) \sum_{x\in \Fp} \sigma\bigg((x-3)^3+27\cdot \frac{J+1}{J-2}\bigg)\\
&= \sigma(J-2) \sum_{x\in \Fp}\sigma\bigg(x^3+27\cdot \frac{J+1}{J-2}\bigg)=\sigma\big(3(J-2)\big) \psi_3\bigg(\frac{J+1}{J-2}\bigg).
\end{align*}

The $N_k$ counts are now easily obtained. The formula for $N_0$ comes from Theorem~\ref{thm: v}(ii) and Lemma~\ref{lem: 67}. We note that $\sigma(D-D_\bullet)=\sigma(J^2-4)$, by \eqref{eq: Jdiff}.

The formula for $N_1$ comes from Theorem~\ref{thm: N1}(ii), coupled with the above evaluation of the associated character sum. The formula for $N_2$ comes from Theorem~\ref{thm: N2}(ii). Finally, by Theorem~\ref{thm: nodes} we have $N_3=0$ since we have discarded any overlaps with the Cayley family or the $K$ family.
\end{proof}

In presenting the degree distributions, we find it convenient to separate according to the value of $p$ mod $3$. In the case $p\equiv 2$ mod $3$ the outcome is especially simple.

\begin{thm}
Let $p\equiv 2$ mod $3$. Consider the Vieta graph of a cubic in the $J$ family, where $J\neq \pm 2, -1$. The total number of vertices, the number of deficient vertices, and the degree distribution for the deficient vertices are given in Table \ref{T6} below.

\begin{table}[ht]\renewcommand{\arraystretch}{1.2}
 \begin{tabular}{| c || c | c |} 
  \hline
 &  $J^2\neq 4/3$ &  $J^2=4/3$   \\
 \hline
 \rule{0pt}{2.5ex}total & \:$p^2+1+2\sigma(J^2-4)p$\:  & \:$p^2+1-2\sigma(2)p$\: \\
deficient &  $3(p-4)$ & $3(p-3)$  \\ \hline
degree $2$ & $3(p-5)$  & $3(p-3)$  \\
degree $1$ & $3$  & $0$ \\
degree $0$ &  $0$  & $0$ \\
 \hline
 \end{tabular} \bigskip\caption{Degree distribution in the Vieta graphs for cubics in the $J$ family, when $p\equiv 2$ mod $3$.}\label{T6}
\end{table}
\end{thm}

\begin{proof}
The degree counts follow from the $N_k$ counts of the previous lemma, keeping in mind that $N_0(\lambda)=1$ and the Jacobsthal sum $\psi_3$ vanishes. When $J^2=4/3$, we have $\sigma(J^2-4)=\sigma(-8/3)=-\sigma(2)$, as $\sigma(-3)=-1$ whenever $p\equiv 2$ mod $3$.
\end{proof}

Two quick comments are in order. Firstly, $J^2=4/3$ has a solution in $\Fp$ if and only if $\sigma(3)=1$; when $p\equiv 2$ mod $3$, this is the case if and only if $p\equiv 3$ mod $4$. Secondly, if $J^2=4/3$, then the lack of vertices of degree $0$ or $1$ agrees with part (i) of Corollary~\ref{cor: no01}. Indeed, $C=0$ and $D=9J(J^2-4)-28=-24J-28=-12(J+1)^2$; now $\sigma(D)=\sigma(-3)=-1$.

The case $p\equiv 1$ mod $3$ can be spelled out as well, but we prefer to highlight a special subfamily.

\begin{thm}
Let $p\equiv 1$ mod $3$. Assume $J\neq \pm 2, -1$, $J^2\neq 4/3$ and $\dfrac{J+1}{J-2}=T^3$ for some $T\in \Fp^*$. For the Vieta graph of the corresponding $J$ cubic, the total number of vertices, the number of deficient vertices, and the degree distribution for the deficient vertices are given in Table \ref{T7} below.
\begin{table}[ht]\renewcommand{\arraystretch}{1.2}
 \begin{tabular}{| c || c |} 
 \hline
\rule{0pt}{2.5ex}total & $p^2+1+4\sigma(1-4T^3)p$  \\
\rule{0pt}{2ex}deficient &  \:$3\big(p-6+2\sigma(T^4-T)A_3(p)\big)$\: \rule[-1.5ex]{0pt}{0pt}   \\ \hline
\rule{0pt}{2.5ex}degree $2$ & $3\big(p-9+2\sigma(T^4-T)A_3(p)\big)$  \\
degree $1$ & $9$   \\
degree $0$ &  $0$  \\
 \hline
 \end{tabular} \bigskip\caption{Degree distribution in the Vieta graphs for cubics in a special $J$ subfamily, when $p\equiv 1$ mod $3$.}\label{T7}
\end{table}
\end{thm}

\begin{proof}
The degree counts follow from the $N_k$ counts of Lemma~\ref{lem: Jfam}, subject to some simplifications. We have $N_0(\lambda)=3$. Next, 
\[\psi_3\bigg(\frac{J+1}{J-2}\bigg)=\psi_3(T^3)=\sigma(T)\psi_3(1)=2\sigma(T)A_3(p).\]
Finally, as $J=(2T^3+1)/(T^3-1)$, we compute
\begin{align*}
\sigma\big(3(J-2)\big)&=\sigma\bigg(\frac{9}{T^3-1}\bigg)=\sigma(T^3-1),\\
\sigma(J^2-4)&=\sigma\bigg(\frac{3(4T^3-1)}{(T^3-1)^2}\bigg)=\sigma(1-4T^3).
\end{align*} 
In the last step, we used the fact that $\sigma(-3)=1$ when $p\equiv 1$ mod $3$. \end{proof}

Degree counting for the $J$ family has become somewhat technical. So let us ground our analysis with a concrete example. 

Consider the parameter value $J=-2/3$, so $C=-8/3$ and $D=-20/3$. We can tabulate degree counts for the Vieta graph of the corresponding cubic as follows.

\begin{table}[ht]\renewcommand{\arraystretch}{1.2}
 \begin{tabular}{| c || c | c |} 
  \hline
 &  $p\equiv 2$ mod $3$ &  $p\equiv 1$ mod $3$   \\
 \hline
 \rule{0pt}{3ex}total & \:$p^2+1+2\sigma(-2)p$\:  & \:$p^2+1+4\sigma(-2)p$\: \\
deficient &  $3(p-4)$ & $3\big(p-6+2A_3(p)\big)$  \\ \hline
degree $2$ & $3(p-5)$  & $3\big(p-9+2A_3(p)\big)$  \\
degree $1$ & $3$  & $9$ \\
degree $0$ &  $0$  & $0$ \\
 \hline
 \end{tabular} \bigskip\caption{Degree distribution in the Vieta graph for the cubic $\M(-8/3, -20/3)$.}
\end{table}

Note that $J^2\neq 4/3$. The column for $p\equiv 2$ mod $3$ draws from Table~\ref{T6}.  The column for $p\equiv 1$ mod $3$ draws from Table~\ref{T7}, by taking $T=-1/2$.

\begin{figure}[ht]
    \centering  
    \includegraphics[width=1.0\textwidth]{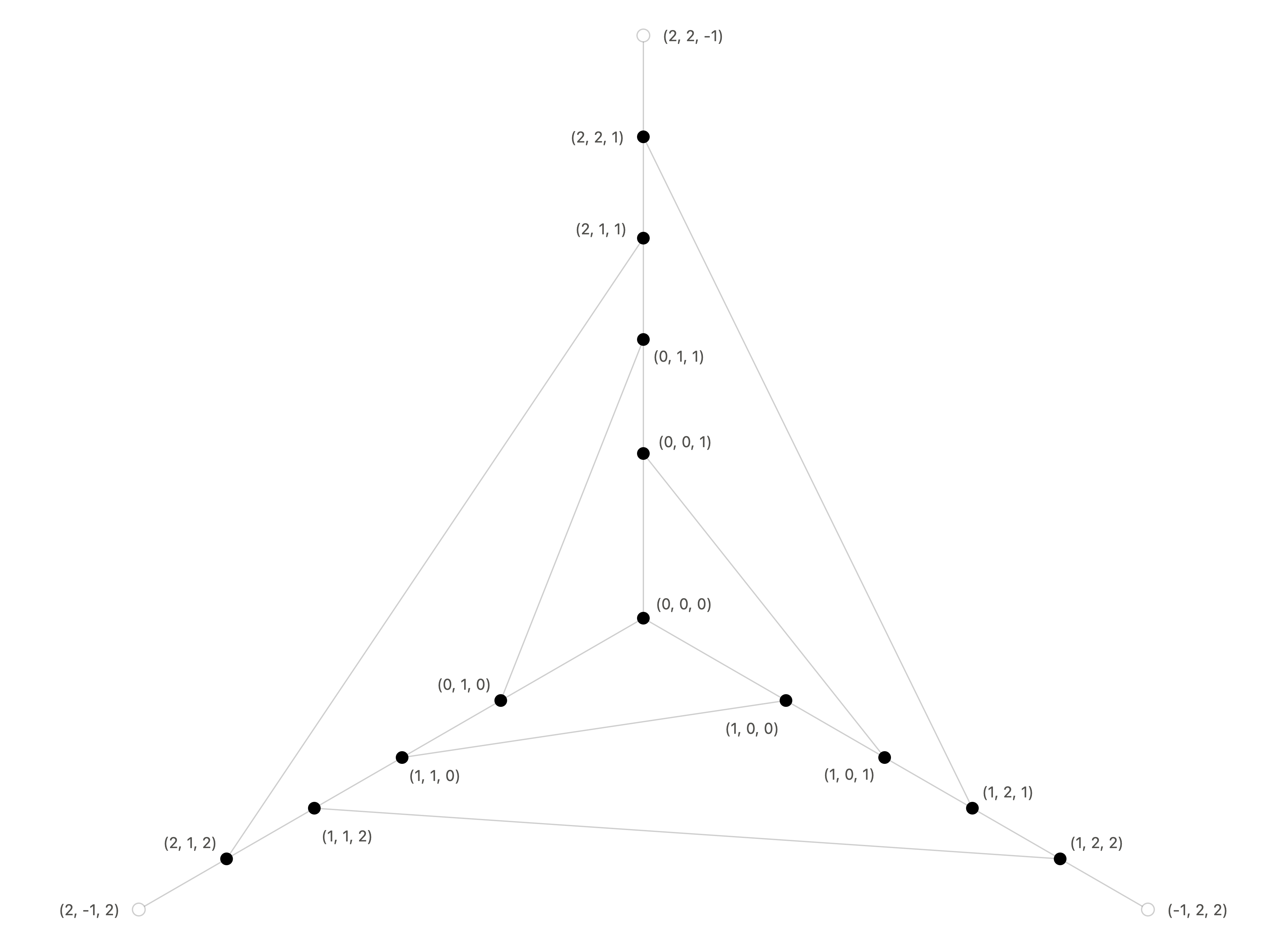}
   \caption{The Vieta graph of the Cayley cubic $\M(-8/3,-20/3)$ over $\F_{\!5}$. The leaf vertices are highlighted in white.}
    \label{fig: Cm1_D0}
\end{figure}

\bigskip
\part{Vieta graphs in dimension $4$}
In this part we discuss two parametric families of Vieta graphs in dimension $4$. But these are just incursions in the largely untouched Vieta landscape in dimension $4$; a more thorough study is needed. As for dimension $5$ or higher, explicit counts seem  difficult to obtain.  

We recall from Lemma~\ref{lem: Nk} that the total number of vertices, the number of deficient vertices, and the degree distribution in a $4$-dimensional Vieta graph can be deduced from the $N_k$ counts--$N_0$, $N_1$, $N_2$, $N_3$, and $N_4$--as follows.

\begin{table}[ht]\renewcommand{\arraystretch}{1.2}
 \begin{tabular}{| c || c |} 
 \hline
total & $N_0$  \\
deficient & $4N_1-6N_2+4N_3-N_4$  \\
 \hline
 degree $3$ & \: $4(N_1-3N_2+3N_3-N_4)$ \: \\
degree $2$ & $6(N_2-2N_3+N_4)$  \\
degree $1$ & $4(N_3-N_4)$ \\
degree $0$ & $N_4$  \\
 \hline
 \end{tabular} 
 \bigskip\caption{Degree counts from $N_k$ counts in dimension $4$.}\label{T8}\end{table}

\section{Cayley-type quartics}
The equation
\begin{align}\label{eq: Cay4}
(x+y+z+t+F)^2=Exyzt
\end{align} 
where $E\neq 0$, defines a two-parameter family of quartic hypersurfaces in $\Fp^4$. Recalling the form \eqref{eq: C'} for the Cayley family, we may think of \eqref{eq: Cay4} as a $4$-dimensional analogue. In fact, for a fixed $t\in \Fp^*$, \eqref{eq: Cay4} can be brought by the uniform rescaling $x:=x/Et$ etc., to \eqref{eq: C'} with parameter $C(t)=2Et(t+F)+8$.

\begin{countlem}\label{thm: Cay4N}
The $N_k$ counts for the polynomial 
\[f(x,y,z,t)=(x+y+z+t+F)^2-Exyzt\] 
are given as follows:
\begin{align*}
N_0(f)&=p^3-1+\sigma(E)\cdot S\Big(\frac{F^2E}{4}-2\Big)\cdot p,\\
N_1(f)&=4p^2-\big(6-\sigma(F^2-16E^{-1})\big)p+5-\llbracket F=0\rrbracket,\\
N_2(f)&=2p^2+\big(1+\sigma(F^2-16E^{-1})\big)(p-1)-2+\llbracket F=0\rrbracket,\\
N_3(f)&=p^2+3p-4+2\llbracket F=0\rrbracket+\sigma(F^2-16E^{-1}),\\
N_4(f)&=6p-8+3\llbracket F=0\rrbracket+\llbracket F^2=16E^{-1}\rrbracket.
\end{align*}
\end{countlem}
Concerning the formula for the $N_0$ count, we recall that the one-parameter character sum $S(a)$ is defined in \eqref{eq: Sa}.

\begin{proof} We apply the Discriminant Lemma~\ref{lem: start!} to $f(x,y,z,t)$. Here
\begin{align*}
\begin{cases}
\alpha_t(x,y,z)=1,\\
\beta_t(x,y,z)=2(x+y+z+F)-Exyz,\\
\gamma_t(x,y,z)=(x+y+z+F)^2,
\end{cases}
\end{align*}
and
\begin{align*}
\Delta_t(x,y,z)&=\beta_t(x,y,z)^2-4\alpha_t(x,y,z)\gamma_t(x,y,z)\\
&=Exyz\big(Exyz-4(x+y+z+F)\big).
\end{align*}

Formula \eqref{eq:N0gen} gives
\[N_0(f)=p^3+\sum_{x,y,z\in \Fp} \sigma(\Delta_t(x,y,z)).\] 
In turn, the latter trivariate character sum can be computed by applying \eqref{eq:Qgen} to the polynomial $\Delta_t(x,y,z)$. Now
 \begin{align*}
\alpha_z(x,y)=Exy(Exy-4),\qquad \beta_z(x,y)=-4Exy(x+y+F),\qquad \gamma_z(x,y)=0.
\end{align*}
So the corresponding discriminant is, quite simply, $\Delta_z(x,y)=\beta_z(x,y)^2$. We obtain
\begin{align*}
\sum_{x,y,z\in \Fp} \sigma(\Delta_t(x,y,z))=-\sum_{x,y\in \Fp} \sigma(\alpha_z(x,y))+p\sum_{x,y\in \Fp: \: \beta_z(x,y)=0} \sigma(\alpha_z(x,y)).
\end{align*}
The complete character sum can be easily evaluated:
\begin{align*}
\sum_{x,y\in \Fp} \sigma\big(Exy(Exy-4)\big)&=\sum_{x\in \Fp^*} \sum_{y\in \Fp} \sigma\big(y^2-4(Ex)^{-1}y\big)=-(p-1).
\end{align*}
We now consider the incomplete character sum, which we temporarily denote by $\Sigma$. We begin by writing
\begin{align*}
\Sigma&=\sum_{x,y\in \Fp: \: \beta_z(x,y)=0} \sigma(\alpha_z(x,y))\\
&=\sum_{x,y\in \Fp:\: x+y+F=0} \sigma\big(Exy(Exy-4)\big)\\
&=\sum_{x\in \Fp} \sigma\Big(Ex(x+F)\big(Ex(x+F)+4\big)\Big)\\
&=\sum_{x\in \Fp} \sigma\Big(x(x+F)\big(x(x+F)+4E^{-1}\big)\Big).
\end{align*}
We now have a univariate character sum, whose argument is the quartic polynomial $x^4+2Fx^3+(F^2+4E^{-1})x^2+4FE^{-1}x$. Using the quartic-to-cubic transformation lemma~\ref{lem: q2c}, we obtain the cubic polynomial
\[x^3+(F^2+4E^{-1})x^2+8F^2E^{-1}x+(4FE^{-1})^2=x^2(x+4E^{-1})+F^2(x+4E^{-1})^2.\]
Thus
\begin{align*}
\Sigma=-1+\sum_{x\in \Fp} \sigma\Big(x^2(x+4E^{-1})+F^2(x+4E^{-1})^2\Big).
\end{align*}
The rescaling $x:=4E^{-1}x$, followed by the shift $x:=x-1$, lead to
\begin{align*}
\Sigma&=-1+\sigma(E)\sum_{x\in \Fp} \sigma\bigg(x^3+\Big(\frac{F^2E}{4}-2\Big)x^2+x\bigg)\\
&=-1+\sigma(E)\cdot S\Big(\frac{F^2E}{4}-2\Big).
\end{align*}
To summarize, we have
\begin{align*}
N_0(f)=p^3+(p-1)+p\Sigma=p^3-1+\sigma(E)\cdot S\Big(\frac{F^2E}{4}-2\Big)\cdot p.
\end{align*}

The evaluation of $N_1(f)$ uses formula \eqref{eq:N1gen}, which gives $N_1(f)=N_0(\Delta_t)$. Counting the zeros of $\Delta_t(x,y,z)$ is now elementary, owing to its factorization. The counts for $N_2(f)$, $N_3(f)$, and $N_4(f)$ are likewise elementary, and somewhat tedious. They are all omitted and left to the reader.
\end{proof}

We note that the calculation of $N_0(f)$ can also be done by fibering the hypersurface \eqref{eq: Cay4} along $t\in \Fp$. Each fiber is a cubic surface of Cayley type, except when $t=0$ for which we get a plane $x+y+z+F=0$. So $N_0(f)$ can be obtained by summing over the fibers. The standardized approach via the Discriminant Lemma is, however, more convenient since it gives a way to compute $N_1(f)$ as well.

In order to make the $N_0$ count explicit, we need to resort to Lemma~\ref{lem: Sa} once again. Imposing $\dfrac{F^2E}{4}-2=0, \pm 2, \pm 6$ amounts to $F^2E=0,8, \pm 16, 32$. When $F\neq 0$ we use the uniform scalings $x:=Fx$ etc. in \eqref{eq: Cay4} to get the normalized form $(x+y+z+t+1)^2=(F^2E)xyzt$. Thus the cases $F^2E=8, \pm 16, 32$ provide four specific quartic hypersurfaces of the form 
\begin{align}
\mathcal{Q}_c: \qquad (x+y+z+t+1)^2=cxyzt,
\end{align} 
for $c=8, \pm 16, 32$. The case $F^2E=0$, that is $F=0$, is the one-parameter family of quartic hypersurfaces
\begin{align}\label{eq: Cay40}
\mathcal{Q}(E): \qquad (x+y+z+t)^2=Exyzt.
\end{align} 
Here too we can rescale uniformly, say $x:=sx$ etc., which leads to the parameter change $E:=s^2E$. So, up to rescaling, $\mathcal{Q}(E)$ only depends on the quadratic signature $\sigma(E)$, and not on $E$ itself. 

The computation of $N_0(f)$ is, essentially, due to Carlitz \cite[Thm.2]{Car2}. But herein we push it a step further, by linking it to the sum $S(a)$. This allows us to single out more parameter values for which $N_0(f)$ is explicitly computable. The cases 
$F=0$ and $F^2E=8, 16$ (in our notation) were already pointed out by Carlitz; to which we may now add, as discussed above, the cases $F^2E=-16, 32$. We remark that the evaluation of $N_0(f)$ for the latter value--that is, the point count for the quartic $\mathcal{Q}_{32}$--is also worked out in \cite[Exer.5.32]{N}. 

We now turn to the Vieta graph of \eqref{eq: Cay4}. The total vertex count, given by $N_0(f)$, is the bottleneck which forces us to focus on specific quartics in order to explicitly compute it. The vertex counts for deficient vertices are general. They follow  from Table~\ref{T8} and the previous counting lemma; they are tabulated in Table~\ref{T9}.   

\begin{table}[ht]\renewcommand{\arraystretch}{1.2}
 \begin{tabular}{| c || c |} 
  \hline
degree $3$ & \: $4\big((p-3)^2-2\e(p-3)+1-\llbracket F=0\rrbracket-\llbracket F^2E=16\rrbracket\big)$ \: \\
degree $2$ & $6\big((1+\e)(p-3)+\llbracket F^2E=16\rrbracket\big)$  \\
degree $1$ & $4\big(p^2-3p+4+\e-\llbracket F=0\rrbracket-\llbracket F^2E=16\rrbracket\big)$ \\
degree $0$ & $6p-8+3\llbracket F=0\rrbracket+\llbracket F^2E=16\rrbracket$  \\
 \hline
 \end{tabular} 
 \bigskip\caption{Vertex counts for deficient vertices in the Vieta graph of the general quartic \eqref{eq: Cay4}, where $\e=\sigma(F^2-16E^{-1})$.}\label{T9}\end{table}

We compile the byproducts of our analysis in the following two theorems.

\begin{thm}
For the Vieta graph of a quartic $\mathcal{Q}(E)$, the total number of vertices and the degree distribution for the deficient vertices are given in Table \ref{T10} below.
\begin{table}[h]\renewcommand{\arraystretch}{1.2}
 \begin{tabular}{| c || c |} 
 \hline
total & $p^3-1-\sigma(E) p$  \\
\hline
degree $3$ & \:$4\big((p-3)^2- 2\sigma(-E)(p-3)\big)$\:  \\
degree $2$ & $6(1+\sigma(-E))(p-3)$  \\
degree $1$ & $4\big(p^2-3p+3+\sigma(-E)\big)$   \\
degree $0$ &  $6p-5$  \\
 \hline
 \end{tabular} \bigskip\caption{Degree distribution in the Vieta graph of a quartic $\mathcal{Q}(E)$.}\label{T10}
\end{table}
\end{thm}

\begin{thm}
Consider a quartic $\mathcal{Q}_c$, where $c=8, \pm 16, 32$. For the corresponding Vieta graph, the total number of vertices and the degree distribution for the deficient vertices are given in Tables \ref{T11} and \ref{T12} below.

\begin{table}[ht]\renewcommand{\arraystretch}{1.2}
 \begin{tabular}{| c || c | c |} 
  \hline
 &  $c=8$ &  $c=16$ \\
 \hline
total & $p^3-1+\sigma(2)\varphi_2(1)p$  & \:$p^3-1-\sigma(-1)p$\: \\
\hline
\rule{0pt}{3ex}degree $3$ & \:$4\big(p-3-\sigma(-1)\big)^2$\:  & $4(p-3)^2$  \\
degree $2$ & $6(1+\sigma(-1))(p-3)$  & $6(p-2)$  \\
degree $1$ & \:$4\big(p^2-3p+4+\sigma(-1)\big)$\:  & $4\big(p^2-3p+3\big)$ \\
degree $0$ &  $6p-8$  & $6p-7$ \\
 \hline
 \end{tabular} \bigskip\caption{Degree distribution in the Vieta graphs of the quartics $\mathcal{Q}_{8}$ and $\mathcal{Q}_{16}$.}\label{T11}
\end{table}

\begin{table}[ht]\renewcommand{\arraystretch}{1.2}
 \begin{tabular}{| c || c | c |} 
  \hline
 &  $c=-16$ &  $c=32$ \\
 \hline
total & $p^3-1+\sigma(2)\varphi_2(1)p$  & $p^3-1+\varphi_2(1)p$ \\
\hline
\rule{0pt}{3ex}degree $3$ & $4\big(p-3-\sigma(2)\big)^2$  & $4\big(p-3-\sigma(2)\big)^2$  \\
degree $2$ & $6(1+\sigma(2))(p-3)$  & $6(1+\sigma(2))(p-3)$  \\
degree $1$ & \:$4\big(p^2-3p+4+\sigma(2)\big)$\:  & \:$4\big(p^2-3p+4+\sigma(2)\big)$\: \\
degree $0$ &  $6p-8$  & $6p-8$ \\
 \hline
 \end{tabular} \bigskip\caption{Degree distribution in the Vieta graphs of the quartics $\mathcal{Q}_{-16}$ and $\mathcal{Q}_{32}$.}\label{T12}
\end{table}
\end{thm}

The count for $c=16$ is especially simple. The counts for $c=8, -16, 32$ can be made more explicit by considering two cases: for $p\equiv 3$ mod $4$, we have $\varphi_2(1)=0$, whereas for $p\equiv 1$ mod $4$ we have $\varphi_2(1)=2A_2(p)$.

\section{Markoff--Hurwitz quartics}
An equation of the form $x_1^2+x_2^2+\ldots+x_n^2=Ex_1x_2\dots x_n$, where $E\neq 0$, is called a Markoff--Hurwitz equation (see, for instance, \cite{Bar2}). The case $n=3$ is, up to a uniform rescaling, the Markoff cubic \eqref{eq: M}. We now discuss the case $n=4$, up to a slight relabeling.

Consider the one-parameter family of quartic hypersurfaces defined by the equation
\begin{align}\label{eq: V4}
x^2+y^2+z^2+t^2=2Exyzt
\end{align} 
where $E\neq 0$. By uniformly rescaling each variable in \eqref{eq: V4}, $x:=sx$ etc. for some $s\neq 0$, we can change the parameter $E$ into $s^2E$; we therefore expect the counting analysis to only dependent on the quadratic signature of $E$. Thus, there are essentially only two $4$-dimensional Vieta graphs for the Markoff--Hurwitz quartics--one for $\sigma(E)=+1$ and one for $\sigma(E)=-1$.

The counting lemma below gives explicit formulas for the $N_k$ counts. 

\begin{countlem}\label{thm: HurwitzN}
The $N_k$ counts for the polynomial
\[f(x,y,z,t)=x^2+y^2+z^2+t^2-2Exyzt\]
are given as follows:
\begin{itemize}
\item[(i)] if $p\equiv 3$ mod $4$, then 
\begin{align*}
N_0(f)&=p^3+5p-1,\\
N_1(f)&=(p-2)^2,\\
N_2(f)&=N_3(f)=N_4(f)=1;
\end{align*}
\item[(ii)] if $p\equiv 1$ mod $4$, then 
\begin{align*}
N_0(f)&=p^3-7p-1+8\sigma(E)A_2(p)\cdot p,\\
N_1(f)&=(p+2)^2+4\sigma(2E)\big(p+A_2(p)^2\big),\\
N_2(f)&=6p-5+4\sigma(2E)(p-1),\\
N_3(f)&=9+8\sigma(2E),\\
N_4(f)&=1.
\end{align*}
\end{itemize}
\end{countlem}

The computation of $N_0(f)$ is, essentially, yet another result of Carlitz \cite[Thm.2]{Car1}. He actually counts the number of solutions to the not-necessarily symmetric equation $Ax^2+By^2+Cz^2+Dt^2=2Exyzt+F$. The counting formula simplifies considerably when $F=0$ and $A=B=C=D=1$. This case is also worked out in \cite[Ex.4.3]{N}. We follow Carlitz in using $2E$, in place of $E$, as the coefficient of the quartic term $xyzt$ in \eqref{eq: V4}. This turns out to be quite convenient.

We break the proof into three parts. We first handle the computations of $N_2(f)$, $N_3(f)$, and $N_4(f)$, which are straightforward. Next, we work out the computation of $N_0(f)$ This is done for completeness, but it highlights once again the systematic approach using the Discriminant Lemma. Finally, using the same lemma we compute $N_1(f)$; this turns out to be more involved.

\begin{proof}[Proof: the $N_2$, $N_3$, and $N_4$ counts] Recall that $N_2$ counts the number of solutions $(x,y,z,t)$ to the system $f=0$, $\partial_z f=0$, $\partial_t f=0$. We have $\partial_t f=2(t-Exyz)$, so $\partial_t f=0$ determines $t=Exyz$. Eliminating $t$, our system $f=0$, $\partial_z f=0$ reduces to $x^2+y^2+z^2=(Exyz)^2$, $z=(Exy)^2z$ as a system in $(x,y,z)$. The latter relation gives $(Exyz)^2=z^2$, and so we may further reduce to the system
\begin{align}\label{eq: sys}
\begin{cases}
x^2+y^2=0,\\
z=(Exy)^2z.
\end{cases}
\end{align} 
We note the trivial solution $(0,0,0)$. When $p\equiv 3$ mod $4$, this is the only solution; thus $N_2=1$ in this case. Since $(0,0,0)$ contributes to the $N_4$ count as well (see below for the conditions), we deduce that $N_2=N_3=N_4=1$.

Assume $p\equiv 1$ mod $4$ for the remainder of the proof. Let $j\in\Fp$ satisfy $j^2=-1$. Then $x^2+y^2=0$ yields $y=\pm jx$; in turn $z=(Exy)^2z$ amounts to $z=(jEx^2)^2z$. The system \eqref{eq: sys} has $2p-1$ solutions with $z=0$. When $z\neq 0$, we must have $jEx^2=\pm 1$, which has $2(1+\sigma(jE))$ solutions $x\in \Fp^*$. For each such $x$ there correspond two values of $y$, and $p-1$ possibilities for $z$. Thus
\[N_2=(2p-1)+2\big(1+\sigma(jE)\big)\cdot 2\cdot (p-1).\]
This is the claimed formula for $N_2$, up to observing that $\sigma(j)=\sigma(2)$. Indeed, using the second supplementary law of quadratic reciprocity, the fact that $(p+1)/2$ is odd, and Euler's formula, we have 
\[\sigma(2)=(-1)^{\frac{p^2-1}{8}}=(-1)^{\frac{p-1}{4}}=j^{\frac{p-1}{2}}=\sigma(j).\]

We now turn to evaluating $N_3$ which, we recall, counts the number of solutions $(x,y,z,t)$ to the system $f=0$, $\partial_y f=0$, $\partial_z f=0$, $\partial_t f=0$. This means that the system \eqref{eq: sys} in $(x,y,z)$ is enhanced by an additional relation:
\begin{align}\label{eq: sys2}
\begin{cases}
x^2+y^2=0,\\
z=(Exy)^2z,\\
y=(Exz)^2y.
\end{cases}
\end{align} 
Now the case $z=0$ yields only the trivial solution $(0,0,0)$. The case $z\neq 0$ forces $y\neq 0$, and the latter relation becomes $Exz=\pm 1$. This means that, compared to the above count for $N_2$, the multiplicity in $z$ is only $2$ in place of $p-1$. We obtain
\[N_3=1+2\big(1+\sigma(jE)\big)\cdot 2\cdot 2=8\big(1+\sigma(jE)\big)+1,\]
in which we replace $\sigma(j)$ by $\sigma(2)$ once again.

Finally, we evaluate $N_4$, which counts the number of solutions $(x,y,z,t)$ to the full system $f=0$, $\partial_x f=0$, $\partial_y f=0$, $\partial_z f=0$, $\partial_t f=0$. We now get the system 
\begin{align}\label{eq: sys3}
\begin{cases}
x^2+y^2=0,\\
z=(Exy)^2z,\\
y=(Exz)^2y,\\
x=(Eyz)^2x.
\end{cases}
\end{align} 
There is no solution to this system other than the trivial solution $(0,0,0)$. Indeed, assume one, whence each one, of $x,y,z$ is non-zero. Then $Exz=\pm 1$ and $Eyz=\pm 1$. It follows that $y=\pm x$, which makes $x^2+y^2=0$ impossible.
\end{proof}

\begin{proof}[Proof: the $N_0$ count] We apply formula \eqref{eq:N0gen} of the Discriminant Lemma~\ref{lem: start!} to the polynomial $f(x,y,z,t)$. Here
\begin{align}\label{eq: MHt}
\alpha_t(x,y,z)=1,\qquad \beta_t(x,y,z)=-2Exyz,\qquad \gamma_t(x,y,z)=x^2+y^2+z^2,
\end{align}
and
\begin{align*}
\Delta_t(x,y,z)&=\beta_t(x,y,z)^2-4\alpha_t(x,y,z)\gamma_t(x,y,z)\\
&=4(Exyz)^2-4(x^2+y^2+z^2).
\end{align*}
Thus
\begin{align}\label{eq: MHN0}
N_0(f)=p^3+\sum_{x,y,z\in \Fp} \sigma(\Delta_t(x,y,z)).
\end{align}
To compute the latter sum, we apply \eqref{eq:Qgen} to $\Delta_t(x,y,z)/4$. We have
 \begin{align}\label{eq: MHz}
\alpha_z(x,y)=(Exy)^2-1,\qquad \beta_z(x,y)=0,\qquad \gamma_z(x,y)=-(x^2+y^2)
\end{align}
and
\begin{align*}
\Delta_z(x,y)&=\beta_z(x,y)^2-4\alpha_z(x,y)\gamma_z(x,y)=4\big((Exy)^2-1\big)(x^2+y^2).
\end{align*}
We find that the trivariate character sum in \eqref{eq: MHz} equals
\begin{align*}
-\sum_{x,y\in \Fp} \sigma\big(\alpha_z(x,y)\big)+p\sum_{x,y\in \Fp: \: \gamma_z(x,y)=0} \sigma\big(\alpha_z(x,y)\big)+p\sum_{x,y\in \Fp: \: \alpha_z(x,y)=0} \sigma\big(\gamma_z(x,y)\big).
\end{align*}
Simple manipulations, left to the reader, yield
\begin{align*}
\sum_{x,y\in \Fp} \sigma\big(\alpha_z(x,y)\big)&=\sum_{x,y\in \Fp} \sigma\big((Exy)^2-1\big)\\
&=p\sigma(-1)-(p-1),
\end{align*}
\begin{align*}
\sum_{x,y\in \Fp: \: \gamma_z(x,y)=0} \sigma\big(\alpha_z(x,y)\big)&=\sum_{x,y\in \Fp:\: x^2+y^2=0} \sigma\big((Exy)^2-1\big)\\
&=\sigma(-1)+\big(1+\sigma(-1)\big) \sum_{x\in \Fp^*} \sigma(x^4+E^{-2}),
\end{align*}
\begin{align*}
\sum_{x,y\in \Fp: \: \alpha_z(x,y)=0} \sigma\big(\gamma_z(x,y)\big)&=\sum_{x,y\in \Fp:\: (Exy)^2-1=0} \sigma\big(-(x^2+y^2)\big)\\
&=2\sigma(-1) \sum_{x\in \Fp^*}\sigma(x^4+E^{-2}).
\end{align*}
Combining our computations, we have at this point that
\[N_0(f)=p^3+p-1+p\cdot \big(1+3\sigma(-1)\big) \sum_{x\in \Fp^*} \sigma(x^4+E^{-2}).\]
Now
\begin{align*}
\sum_{x\in \Fp^*} \sigma(x^4+E^{-2})&=-1+\sum_{x\in \Fp} \sigma(x^4+E^{-2})\\
&=-1+\sum_{u\in \Fp} \sigma(u^2+E^{-2})\big(1+\sigma(u)\big)\\
&=-1+\sum_{u\in \Fp} \sigma(u^2+E^{-2})+\sum_{u\in \Fp} \sigma(u^3+E^{-2}u)\\
&=-2+\varphi_2(E^{-2})=-2+\sigma(E)\varphi_2(1).
\end{align*}
The degree-lowering trick used above is to set $u=x^2$, and to note that for each $u$ there are $1+\sigma(u)$ possible values of $x$. We will use this trick in the upcoming $N_1$ count as well.  All in all,
\begin{align*}
N_0(f)&=p^3+p-1+p\cdot \big(1+3\sigma(-1)\big) \big(-2+\sigma(E)\varphi_2(1)\big)\\
&=\begin{cases}
p^3+p-1+4p & \textrm{ if } p\equiv 3 \textrm{ mod }4,\\
p^3+p-1+8p\big(\sigma(E)A_2(p)-1\big) & \textrm{ if } p\equiv 1 \textrm{ mod }4.
\end{cases}
\end{align*}
The claimed formulas for $N_0(f)$ follow.  
\end{proof}

\begin{proof}[Proof: the $N_1$ count] We apply formula \eqref{eq:N1gen}, keeping the notations \eqref{eq: MHt}. As $\alpha_t(x,y,z)$ is constant, we have 
\[N_1(f)=N_0(\Delta_t).\]  
Next, we apply \eqref{eq:N0gen} to $\Delta_t(x,y,z)$; again, we keep the notations \eqref{eq: MHz}. We get
\begin{align}\label{eq: MHN0D}
N_0(\Delta_t)=p^2-N_0(\alpha_z)+pN_0(\alpha_z,\gamma_z)+\sum_{x,y\in \Fp}\sigma\big(\Delta_z(x,y)\big)
\end{align}
since $\beta_z\equiv 0$. Now $N_0(\alpha_z)$ counts the number of solutions $(x,y)$ to $(Exy)^2=1$; we see that 
\begin{align}\label{eq: MHN00}
N_0(\alpha_z)=2(p-1).
\end{align} 
Next, $N_0(\alpha_z,\gamma_z)$ counts the number of solutions $(x,y)$ to the system $(Exy)^2=1$, $x^2+y^2=0$; recalling our previous discussion of the $N_2$ count, we see that 
\begin{align}\label{eq: MHN000}
N_0(\alpha_z,\gamma_z)=\begin{cases}
0 & \textrm{ if } p\equiv 3 \textrm{ mod } 4,\\
4\big(1+\sigma(2E)\big) & \textrm{ if } p\equiv 1 \textrm{ mod } 4.
\end{cases}
\end{align}
The bivariate character sum requires the most work. Using the degree-lowering trick encountered before, but this time in each variable $x$ and $y$, we write
\begin{align*}
\sum_{x,y\in \Fp}\sigma\big(\Delta_z(x,y)\big)&=\sum_{x,y\in \Fp}\sigma\big((E^2x^2y^2-1)(x^2+y^2)\big)\\
&=\sum_{x,y\in \Fp}\sigma\big((E^2xy-1)(x+y)\big)\big(1+\sigma(x)\big)\big(1+\sigma(y)\big)\\
&=\sigma(E)\sum_{x,y\in \Fp}\sigma\big((xy-1)(x+y)\big)\big(1+\sigma(Ex)\big)\big(1+\sigma(Ey)\big)
\end{align*}
by rescaling $x:=x/E$, $y:=y/E$ in the last step. Setting
\begin{align*}
S_0&= \sum_{x,y\in \Fp}\sigma\big((xy-1)(x+y)\big),\\
S_1&= \sum_{x,y\in \Fp}\sigma\big((xy-1)(x+y)\big)\sigma(x),\\
S_2&= \sum_{x,y\in \Fp}\sigma\big((xy-1)(x+y)\big)\sigma(xy),
\end{align*}
we obtain
\begin{align}\label{eq: S}
\sum_{x,y\in \Fp}\sigma\big(\Delta_z(x,y)\big)=\sigma(E) S_0+2S_1+\sigma(E) S_2.
\end{align}
The computations of the three sums, $S_0$, $S_1$, and $S_2$, share some similarities. To start off,
\begin{align*}
S_0&= \sum_{y\in \Fp} \sigma(-y)+\sum_{x\in \Fp^*}\sigma(x)\sum_{y\in \Fp} \sigma\big((y-x^{-1})(y+x)\big)\\
&=\sum_{x\in \Fp^*}\sigma(x)\big(p \llbracket -x^{-1}=x\rrbracket-1 \big)=p\sum_{x\in \Fp^*}\sigma(x) \llbracket x^2=-1\rrbracket.
\end{align*}
There are no solutions to $x^2=-1$ when $p\equiv 3$ mod $4$, so the latter sum vanishes in this case. When $p\equiv 1$ mod $4$, let $j\in \Fp$ satisfy $j^2=-1$. Then the inner sum has two terms, corresponding to $x=j$ and $x=-j$, and it evaluates to $\sigma(j)+\sigma(-j)=2\sigma(j)=2\sigma(2)$. Overall, 
\begin{align}\label{eq: S0}
S_0=\begin{cases}
0 & \textrm{ if } p\equiv 3 \textrm{ mod } 4,\\
\sigma(2)\cdot 2p & \textrm{ if } p\equiv 1 \textrm{ mod } 4.
\end{cases}
\end{align}

Next, 
\begin{align*}
S_1&= \sum_{x\in \Fp^*} \sum_{y\in \Fp} \sigma\big((y-x^{-1})(y+x)\big)\\
&=\sum_{x\in \Fp^*}\big(p \llbracket -x^{-1}=x\rrbracket-1 \big)=-(p-1)+p\sum_{x\in \Fp^*}\llbracket x^2=-1\rrbracket
\end{align*}
whence
\begin{align}\label{eq: S1}
S_1&=\begin{cases}
-p+1 & \textrm{ if } p\equiv 3 \textrm{ mod } 4,\\
p+1 & \textrm{ if } p\equiv 1 \textrm{ mod } 4.
\end{cases}
\end{align}

Finally, we claim that
\begin{align}\label{eq: S2}
S_2&=\begin{cases}
0 & \textrm{ if } p\equiv 3 \textrm{ mod } 4,\\
\sigma(2)(4A_2(p)^2-2p) & \textrm{ if } p\equiv 1 \textrm{ mod } 4.
\end{cases}
\end{align}
Indeed, when $p\equiv 3$ mod $4$ we observe that the double sign change $x:=-x$, $y:=-y$ yields $S_2=\sigma(-1) S_2=-S_2$, whence $S_2=0$. The same idea could have been used to show $S_0=0$.

Assume now that $p\equiv 1$ mod $4$. We have
\begin{align*}
\varphi_2(1)\varphi_2(-1)&=\bigg(\sum_{x\in \Fp} \sigma(x^3+x)\bigg) \bigg(\sum_{y\in \Fp} \sigma(y^3-y)\bigg)\\
&=\sum_{x,y\in \Fp} \sigma\big(xy(x^2+1)(y^2-1)\big).
\end{align*}
Recall that $\varphi_2(1)=2A_2(p)$ and $\varphi_2(-1)=\sigma(2)\cdot 2A_2(p)$. In the latter double sum, we may restrict to summing over $x,y\neq 0$, and we make the change of variable $x:=x/y$. We then get, by the degree-lowering trick
\begin{align*}
\sigma(2)\cdot 4A_2(p)^2&=\sum_{x,y\in \Fp^*} \sigma\big(x(x^2+y^2)(y^2-1)\big)\\
&=\sum_{x, y\in \Fp^*} \sigma\big(x(x^2+y)(y-1)\big) \big(1+\sigma(y)\big).
\end{align*} 
We peel off and compute the following sum:
\begin{align*}
\sum_{x, y\in \Fp^*} \sigma\big(x(x^2+y)(y-1)\big)&=\sum_{x\in \Fp^*} \sigma(x)\sum_{y\in \Fp^*}  \sigma\big((y+x^2)(y-1)\big)\\
&=\sum_{x\in \Fp^*} \sigma(x)\Big(\big(p\llbracket x^2=-1\rrbracket-1\big)-\sigma(-x^2) \Big)\\
&=p \sum_{x\in \Fp^*} \sigma(x)\llbracket x^2=-1\rrbracket.
\end{align*}
As we have already seen, the latter expression evaluates to $\sigma(2)\cdot 2p$.

Therefore
\begin{align*}
\sigma(2)\big(4A_2(p)^2-2p\big)&=\sum_{x, y\in \Fp^*} \sigma\big(xy(x^2+y)(y-1)\big)\\
&=\sum_{x, y\in \Fp^*} \sigma\big(x(xy)(x^2+xy)(xy-1)\big)\\
&=\sum_{x, y\in \Fp^*} \sigma\big(xy(x+y)(xy-1)\big).
\end{align*}
On the way, we effected the change of variable $y:=xy$. But the latter double sum equals $S_2$, since we can extend to summing over $x,y\in \Fp$. This completes the verification of \eqref{eq: S2}.

Inserting the computations \eqref{eq: S0}, \eqref{eq: S1}, \eqref{eq: S2} into \eqref{eq: S} we get
\begin{align}\label{eq: MHss}
\sum_{x,y\in \Fp}\sigma\big(\Delta_z(x,y)\big)&=\begin{cases}
-2p+2 & \textrm{ if } p\equiv 3 \textrm{ mod } 4,\\
2p+2+\sigma(2E)\cdot 4A_2(p)^2& \textrm{ if } p\equiv 1 \textrm{ mod } 4.
\end{cases}
\end{align}

By combining \eqref{eq: MHN00}, \eqref{eq: MHN000}, and \eqref{eq: MHss} into \eqref{eq: MHN0D}, we conclude that
\begin{align*}
N_1(f)=N_0(\Delta_t)&=\begin{cases}
(p-2)^2 & \textrm{ if } p\equiv 3 \textrm{ mod } 4,\\
(p+2)^2+4\sigma(2E)\big(p+A_2(p)^2\big)& \textrm{ if } p\equiv 1 \textrm{ mod } 4.
\end{cases}
\end{align*}
This completes the proof of Theorem~\ref{thm: HurwitzN}.
\end{proof}

In the $N_1$ count above, we need to count the solutions to equation $\Delta_t(x,y,z)=0$, that is to say, 
\[(E^2x^2y^2-1)z^2=x^2+y^2.\] 
This is a close relative of the solution-counting for the equation
\begin{align*}
z^2=(x^2+y^2)(1+x^2y^2)
\end{align*}
that is worked out in \cite[Exer.4.20]{N}. Compare \cite[Thm.4.5]{KTVZ} and \cite[Thm.2]{U}.

From the previous counting lemma, we may deduce the following. 

\begin{thm}
For the Vieta graph of a Markoff--Hurwitz quartic \eqref{eq: V4}, the total number of vertices and the degree distribution for the deficient vertices are given in Table \ref{T13} below.
\begin{table}[ht]\renewcommand{\arraystretch}{1.2}
\begin{tabular}{| c || c | c |} 
  \hline
 &  $p\equiv 3$ mod $4$ &  $p\equiv 1$ mod $4$ \\
 \hline
\rule{0pt}{2.5ex}total & $p^3+5p-1$  & $p^3-7p-1+8\sigma(E)A_2(p)\cdot p$ \\
\hline
\rule{0pt}{2.5ex}degree $3$ & \:$4(p^2-4p+3)$\:  & \:$4\big(p^2-14p+45+4\sigma(2E)(A_2(p)^2-2p+9)\big)$\:  \\
degree $2$ & $0$  & $6\big(6p-22+4\sigma(2E)(p-5)\big)$  \\
degree $1$ & $0$  & $32\big(1+\sigma(2E)\big)$ \\
degree $0$ &  $1$  & $1$ \\
 \hline
 \end{tabular} \bigskip\caption{Degree distribution in the Vieta graph of a Markoff--Hurwitz quartic.}\label{T13}
\end{table}
\end{thm}

\addtocontents{toc}{\protect\vspace{1em}}

\end{document}